\documentclass[11pt, reqno]{amsart}
\usepackage[utf8]{inputenc} 

\usepackage{geometry} 

\usepackage{fullpage}
\usepackage{graphicx} 

\usepackage{amsmath}
\usepackage{amssymb}
\usepackage{amsfonts}
\usepackage{amsthm}
\usepackage{mathrsfs} 
\usepackage{multirow}
\usepackage{relsize}
\usepackage{tikz-cd}
\usepackage{tikz}
\usepackage{tensor}
\usepackage{geometry}
\usepackage{textpos}
\usepackage{bm}
\usepackage{cleveref}

\usepackage{booktabs} 
\usepackage{array} 
\usepackage{paralist} 
\usepackage{verbatim} 
\usepackage{enumitem}
\usepackage{subfig} 

\usepackage{lscape}
\usepackage{float}

\usepackage{thmtools}
\usepackage{thm-restate}

\newcommand{\leg}[2]{\left(\frac{#1}{#2}\right)}

\newtheorem{thm}{Theorem}
\newtheorem{lem}[thm]{Lemma}
\newtheorem{cor}[thm]{Corollary}
\newtheorem{prp}[thm]{Proposition}

\newtheorem*{rmk}{Remark}
\newtheorem*{remark*}{Remark}

\setlist[itemize]{leftmargin=*}

\theoremstyle{definition}

\theoremstyle{remark}

\newtheorem*{example}{\bf Example}
\numberwithin{thm}{section}
\numberwithin{prp}{section}
\numberwithin{lem}{section}
\numberwithin{cor}{section}
\numberwithin{equation}{section}
\numberwithin{conjecture}{section}

\crefname{prp}{proposition}{propositions}

\newcommand{\Z}{\mathbb{Z}}
\newcommand{\Q}{\mathbb{Q}}
\newcommand{\N}{\mathbb{N}}
\newcommand{\R}{\mathbb{R}}
\newcommand{\C}{\mathbb{C}}

\newcommand{\Qc}{\mathcal{Q}}
\newcommand{\Dc}{\mathcal{D}}
\newcommand{\Ec}{\mathcal{E}}
\newcommand{\Nc}{\mathcal{N}}
\newcommand{\Rc}{\mathcal{R}}
\newcommand{\Cc}{\mathcal{C}}

\newcommand{\Ac}{\mathcal{A}}
\newcommand{\Sc}{\mathcal{S}}
\newcommand{\Tc}{\mathcal{T}}

\newcommand{\Uc}{\mathcal{U}}

\renewcommand{\a}{\alpha}
\renewcommand{\b}{\beta}

\renewcommand{\d}{\delta}
\newcommand{\e}{\varepsilon}

\renewcommand{\l}{\lambda}
\newcommand{\om}{\omega}

\renewcommand{\th}{\theta}
\newcommand{\z}{\zeta}
\newcommand{\Ga}{\Gamma}
\newcommand{\La}{\Lambda}

\newcommand{\bs}{\backslash}
\newcommand{\cb}[1]{\left\{{#1}\right\}}
\newcommand{\cbd}[2]{\left\{{#1}\; :\; {#2}\right\}}

\newcommand{\Log}{\operatorname{Log}}

\newcommand{\im}{\operatorname{Im}}

\renewcommand{\pmod}[1]{\ \left( \mathrm{mod} \, #1 \right)}
\newcommand{\Pmod}[1]{\ ( \mathrm{mod} \, #1 )}

\newcommand{\flo}[1]{\lfloor #1\rfloor}
\newcommand{\Flo}[1]{\left\lfloor #1\right\rfloor}
\newcommand{\cei}[1]{\lceil #1\rceil}

\newcommand{\rb}[1]{\left({#1}\right)}
\newcommand{\re}{\operatorname{Re}}

\newcommand{\sbe}{\subseteq}

\newcommand{\vb}[1]{\left| {#1} \right|}

\newcommand{\Arg}{\operatorname{Arg}}

\newcommand{\Li}{\operatorname{Li}}

\newcommand{\mc}{\mathcal}
\newcommand{\mf}{\mathfrak}

\newcommand{\sm}{\setminus}

\allowdisplaybreaks

\newcommand{\bdd}{\begin{center}\begin{tikzcd}}
\newcommand{\bd}{\begin{tikzcd}}
\newcommand{\edd}{\end{tikzcd}\end{center}}
\newcommand{\ed}{\end{tikzcd}}
\newcommand{\bdp}{\begin{center}\begin{tikzpicture}}
\newcommand{\edp}{\end{tikzpicture}\end{center}}
\newcommand{\bi}{\begin{itemize}}
\newcommand{\ei}{\end{itemize}}
\newcommand{\bt}{\begin{tikzpicture}}
\newcommand{\et}{\end{tikzpicture}}
\newcommand{\ba}{\[\begin{aligned}}
\newcommand{\ea}{\end{aligned}\]}
\newcommand{\bp}{\begin{pmatrix}}
\newcommand{\ep}{\end{pmatrix}}
\newcommand{\bsm}{\begin{smallmatrix}}
\newcommand{\esm}{\end{smallmatrix}}
\newcommand{\bv}{\begin{vmatrix}}
\newcommand{\ev}{\end{vmatrix}}
\newcommand{\bb}{\begin{bmatrix}}
\newcommand{\eb}{\end{bmatrix}}
\newcommand{\bB}{\begin{Bmatrix}}
\newcommand{\eB}{\end{Bmatrix}}
\newcommand{\bea}{\begin{enumerate}[leftmargin=*,label=\textnormal{(\alph*)}]}
\newcommand{\ber}{\begin{enumerate}[leftmargin=*,label=\textnormal{(\roman*)}]}
\newcommand{\ben}{\begin{enumerate}[leftmargin=*,label=\textnormal{(\arabic*)}, wide, labelindent=0pt, labelwidth=*]}
\newcommand{\ee}{\end{enumerate}}

\newcommand{\hlc}{\color{cyan!80!black}}

\makeatletter
\renewcommand{\boxed}[1]{\text{\fboxsep=.2em\fbox{\m@th$\displaystyle#1$}}}
\makeatother

\title{Asymptotics of parity biases for partitions into distinct parts via Nahm sums}
\author{Kathrin Bringmann}
\author{Siu Hang Man}
\address{University of Cologne, Department of Mathematics and Computer Science, Weyertal 86-90, 50931 Cologne, Germany}
\email{kbringma@math.uni-koeln.de}
\email{sman1@math.uni-koeln.de}

\address{Charles University, Faculty of Mathematics and Physics, Department of Algebra, Sokolovská 49/83, 186 75 Praha 8, Czechia}
\email{shman@karlin.mff.cuni.cz}

\author{Larry Rolen}
\address{Department of Mathematics, 1420 Stevenson Center, Vanderbilt University, Nashville, TN 37240}
\email{larry.rolen@vanderbilt.edu}

\author{Matthias Storzer}
\address{Max Planck Institute for Mathematics, Vivatsgasse 7, 53111 Bonn, Germany}
\email{storzer@mpim-bonn.mpg.de}

\begin{document}

\begin{abstract}
	For a random partition, one of the most basic questions is: what can one expect about the parts which arise? For example, what is the distribution of the parts of random partitions modulo $N$? Since most partitions contain a $1$, and indeed many $1$s arise as parts of a random partition, it is natural to expect a skew towards $1\Pmod N$. This is indeed the case. For instance, Kim, Kim, and Lovejoy recently established ``parity biases'' showing how often one expects partitions to have more odd than even parts. Here, we generalize their work to give asymptotics for biases $\Pmod N$ for partitions into distinct parts. The proofs rely on the Circle Method and give independently useful techniques for analyzing the asymptotics of Nahm-type $q$-hypergeometric series. 
\end{abstract}
\maketitle


\section{Introduction and statement of results}

The study of the distribution of parts in integer partitions has a long history. For instance, many authors have studied counting functions for partitions with restrictions on the possible parts. Here, we are interested in the relative numbers of parts which are in different congruence classes. As most of the parts in most partitions tend to be small, for example, including many $1$s and $2$s, it is natural to expect that the parts are not equidistributed modulo $N\in\N_{\ge2}$. For instance, Kim, Kim, and Lovejoy \cite{KKL2020} showed the following parity bias for partitions modulo $2$, where $p_o(n)$ (resp. $p_e(n)$) denotes the number of partitions of $n$ with a majority of odd parts (resp. even parts):
\begin{equation}\label{eq:pepo} 
	p_o(n)>p_e(n) \quad \text{ for } n\neq2,\quad\quad \lim_{n\rightarrow\infty}\frac{p_o(n)}{p_e(n)}=1+\sqrt2. 
\end{equation}

There are also biases for congruence classes modulo general $N$. A deep analysis of these was given in three related papers by Dartyge, Sarkozy, and Szalay \cite{DS,DSS1,DSS2}, who proved lower bounds on biases for parts of partitions in congruence classes for positive proportions of partitions. This was generalized by Beckwith and Mertens \cite{BW}, who turned these results into asymptotic formulas. Recently, the study of partition part biases for partitions into distinct parts was initiated by Kim, Kim, and Lovejoy \cite{KKL2020}. They conjectured an analogue of the inequality in \eqref{eq:pepo} which was recently proven in \cite{BBDMS}. The goal of this paper is to refine these inequalities to precise, general, asymptotics.

For $N,\ell,b\in\N$, the generating function for the number $a_{N,\ell,b}(n)$ of partitions into $\ell$ distinct parts such that the size of each part is at least $b$ and is congruent to $b\Pmod N$ is given by
\[
	\sum_{n\ge0} a_{N,\ell,b}(n)q^n = \frac{q^{\frac{N\ell(\ell-1)}{2}+b\ell}}{\left(q^N;q^N\right)_\ell}.
\]
Here $(a;q)_r:=\prod_{j=0}^{r-1}(1-aq^j)$ with $r\in\N_0\cup\{\infty\}$ is the usual {\it$q$-Pochhammer symbol}.

Let $N\in\N_{\ge2}$, $K\in\N_0$, $1\le\a,\b\le N$, and $\a\ne\b$. Denote by $d_{\a,\b;N}^{[K]}(n)$ the number of partitions of $n$ into distinct parts such that there are more parts of size congruent to $\a\Pmod N$ than parts of size congruent to $\b\Pmod N$ and such that the size of all parts are greater than $K$. The generating function of $d_{\a,\b;N}^{[K]}(n)$ is given by (throughout we use bold letters for vectors)
\begin{equation}\label{eq:ggf} 
	\mc D_{\alpha,\beta;N}^{[K]}(q) := \sum_{n\geq0} d_{\alpha,\beta;N}^{[K]}(n)q^n = \sum_{\substack{\bm n\in\mc S_{\alpha,\beta}}} \frac{q^{N\cdot H(\bm n)}}{\prod_{j=1}^N \rb{q^N;q^N}_{n_j}},
\end{equation}
where
\[
	\Sc_{\a,\b} := \left\{\bm n=(n_1,\dots,n_N)^T\in\N_0^N : n_\a>n_\b\right\}.
\]
Moreover $H\colon\Z^N\to\Q$ is given by
\begin{equation*}
	H(\bm n) := \tfrac12\bm n^T\bm n + \bm b^T\bm n,\qquad \bm b := \bp\frac1N-\frac12, &\frac2N-\frac12, &\ldots, &\frac12\ep^T + \bm e
\end{equation*}
where
\begin{equation*}
	\bm e = \bm e^{[N,K]} = (e_1,\dots,e_N)^T := \Flo{\tfrac KN}\bm1 + (1,\dots,1,0,\dots,0)^T
\end{equation*}
is such that $\sum_je_j=K$ and $\bm1:=(1,\dots,1)^T\in\Z^N$. Note that we have
\begin{equation}\label{eq:sbj} 
	\sum_{j=1}^N b_j = K+\tfrac 12.
\end{equation}

In this paper we investigate a generalized parity bias modulo $N$ for partitions into distinct parts. Namely, given two congruence classes $\a,\b\Pmod N$, we consider the number of partitions of $n$ into distinct parts with more parts congruent to $\a\Pmod N$ than to $\b\Pmod N$, and vice versa. We study the difference between these two counts, which has generating function
\[
\sum_{n\geq0} \rb{d_{\alpha,\beta;N}^{[K]}(n) - d_{\beta,\alpha;N}^{[K]}(n)} q^n.
\]
The case $N=2$, $\a=1$, $\b=2$, and $K=0$ corresponds to the parity bias problem for partitions into distinct parts which Kim, Kim, and Lovejoy considered in \cite{KKL2020} (the reader is also referred to \cite{Chern2022,KK2021} for related works on partition parity biases modulo $N$). Our main result is {as follows}.

\begin{thm}\label{thm:main}
	We have
	\begin{multline*}
		d_{\alpha,\beta;N}^{[K]}(n) = \frac{e^{\pi\sqrt{\frac n3}}}{2^{K+3}\cdot3^\frac14N^{N-1}n^\frac34}\\
		\times \sum_{\substack{\bm\ell\in(\Z/N\Z)^N\\ NH(\bm\ell)\equiv n\pmod{N}}} \rb{1+\frac{N^2-2{ N}[\ell_\alpha-\ell_\beta]_N+\beta-\alpha+N\rb{e_\beta-e_\alpha}}{2\cdot3^\frac14\sqrt{N}n^\frac14} + O\rb{\frac{1}{\sqrt{n}}}}
	\end{multline*}
	as $n\to\infty$, where $[\ell]_N$ denotes the smallest positive integer congruent to $\ell\pmod{N}$.
\end{thm}

\begin{remark*}
	Sums of the shape \eqref{eq:ggf}, i.e., sums over (partial) lattices of $q$ raised to quadratic polynomials divided by products of Pochhammer symbols, have been the subject of many recent works. In particular, they are important in knot theory, algebraic $K$-theory, physics, and the intersections of these subjects with quantum modular forms (see e.g. \cite{CGZ,GZqknots,KS,Nahm}, a more complete list of references and such connections is given in \cite{GZ2021}). They also arise (thought not always named ``Nahm sum'') in $q$-series and combinatorics, for example, in \cite{Andrews, KKL2020}. Recently, Garoufalidis and Zagier \cite{GZ2021} investigated a general class of Nahm sums, but this does not allow us to study \eqref{eq:ggf} particularly due to more general subset of a lattice which we consider. Our analytic methods extend their methods, and are analytically flexible. Thus, they may be more broadly useful in the study of such sums when they arise in applications to combinatorics or knot theory. 
\end{remark*}

If $N=2$ or $N\ge 5$, then we can further simplify the asymptotic formula. In particular, in this case the asymptotic formula does not depend on the congruence class of $n\pmod{N}$. 

\begin{thm}\label{thm:main2} 
Suppose that $N=2$ or $N\ge5$. Then we have, as $n\to\infty$,
\begin{equation*}
d_{\alpha,\beta;N}^{[K]}(n) = \frac{e^{\pi\sqrt{\frac n3}}}{2^{K+3}\cdot 3^{\frac 14} n^\frac34} \rb{1+\frac{{ -N}+\beta-\alpha+N\rb{e_\beta-e_\alpha}}{2\cdot3^\frac14\sqrt{N}n^\frac14} + O\rb{\frac{1}{\sqrt{n}}}}.
\end{equation*}
\end{thm}

Kim, Kim, and Lovejoy \cite[Section 6]{KKL2020} conjectured that for $n\ge20$ there are more partitions into distinct parts with more odd parts than even parts than vice versa. The conjecture was first proved \cite{BBDMS} using combinatorial arguments, but an asymptotic formula for the parity bias had not been found. \Cref{thm:main2} gives a new asymptotic formula for the parity bias as a corollary. 

\begin{cor}\label{cor:pb}
	We have, as $n\to\infty$,
	\ba
		d_{1,2;2}^{[K]}(n) - d_{2,1;2}^{[K]}(n) = \frac{(-1)^Ke^{\pi\sqrt{\frac n3}}}{2^{K+3}\sqrt{6}n} \rb{1+O\rb{n^{-\frac14}}}.
	\ea
	In particular, we have, as $n\to\infty$,
	\ba
		\frac{d_{1,2;2}^{[K]}(n)}{d_{2,1;2}^{[K]}(n)} \to 1.
	\ea
\end{cor}

\begin{rmk}
	\Cref{cor:pb} answers Problem 6.1 of \cite{BBDMS} on explicit inequalities between these quantities (after a small correction). Problem 6.1 of \cite{BBDMS} conjectured that for all $n\ge14$:
	\[
		\begin{cases}
			d_{1,2;2}^{[1]}(n) > d_{2,1;2}^{[1]}(n) & \text{if }n\equiv0\Pmod2,\\
			\vspace{-.4cm}\\
			d_{1,2;2}^{[1]}(n) < d_{2,1;2}^{[1]}(n) & \text{if }n\equiv1\Pmod2.
		\end{cases}
	\]
	While this conjecture is true for many small values, it does not hold in general. \Cref{cor:pb} provides a modified version, showing that $d_{1,2;2}^{[1]}(n)<d_{2,1;2}^{[1]}(n)$ for $n$ sufficiently large. 
\end{rmk}

For $N\in\{3,4\}$, the asymptotics of the parity bias can indeed depend on the residue classes $n\Pmod N$. We use the following corollary to \Cref{thm:main} to illustrate this phenomenon.

\begin{cor}\label{cor:N3}
We have the asymptotic formulas, as $n\to\infty$,
\ba
d_{1,2;3}^{[0]}(n) - d_{2,1;3}^{[0]}(n) =
\begin{cases}
\frac{e^{\pi\sqrt{\frac n3}}}{n} \rb{\frac 1{24}+O\rb{n^{-\frac 14}}} & \text{ if } n\equiv 0 \pmod{3},\\ \frac{e^{\pi\sqrt{\frac n3}}}{n} \rb{\frac 16+O\rb{n^{-\frac 14}}} & \text{ if } n\equiv 1 \pmod{3},\\ \frac{e^{\pi\sqrt{\frac n3}}}{n} \rb{-\frac 1{12} + O\rb{n^{-\frac 14}}} & \text{ if } n\equiv 2 \pmod{3}.
\end{cases}
\ea
\end{cor}

The main tool for the proof of \Cref{thm:main} is the Circle Method. We briefly outline our strategy. Firstly, we look at individual summands in \eqref{eq:ggf}, and derive two different asymptotic series for them as $q$ approaches a root of unity. The first such series is obtained by considering the summands as holomorphic functions in $z$ ($q=e^{-z}$), and deriving an asymptotic expansion as $z\to0$ from the right half-plane. Unfortunately, the asymptotic expansion only holds for a narrow cone around the positive real line and is insufficient for the computation of the major arc contribution. Thus we derive a second asymptotic series by treating the real and imaginary parts of $z$ separately. This asymptotic series is no longer holomorphic in $z$, but holds in a sufficiently large region. To derive an asymptotic expansion for $\Dc_{\a,\b;N}^{[K]}(q)$, we sum over the asymptotic series for the summands in \eqref{eq:ggf}, ignoring the summands with negligible contribution. An asymptotic expansion for $d_{\a,\b;N}^{[K]}(n)$ is then obtained using the Circle Method. \Cref{thm:main} is derived by explicit evaluation of the first two terms in the asymptotic expansion for $d_{\a,\b;N}^{[K]}(n)$.

The paper is organized as follows. In \Cref{section:prelim}, we give preliminaries on asymptotics, such as the asymptotic formula for $-\Log((q;q)_\infty)$ and the Euler--Maclaurin summation formula. In \Cref{section:afs}, we derive the two asymptotic expansions for summands appearing in \eqref{eq:ggf}. In \Cref{section:error}, we prove bounds for the terms we ignore in the derivation of the asymptotic expansion of $\Dc_{\a,\b;N}^{[K]}(q)$, as well as the minor arc contributions. In \Cref{section:afp}, we use the results obtained in previous sections to derive asymptotic expansions for $\Dc_{\a,\b;N}^{[K]}(q)$ and $d_{\a,\b;N}^{[K]}(n)$. In \Cref{section:pot}, we evaluate the asymptotic expansion of $d_{\a,\b;N}^{[K]}(n)$ explicitly, and prove Theorems \ref{thm:main} and \ref{thm:main2}. In \Cref{section:nx}, we give numerical examples illustrating the biases we prove.

\section*{Acknowledgments}

The first and the second author have received funding from the European Research Council (ERC) under the European Union’s Horizon 2020 research and innovation programme (grant agreement No. 101001179). The second author was also supported by the Czech Science Foundation GAČR grant 21-00420M, and the OP RDE project No. CZ.02.2.69/0.0/0.0/18\_053/0016976 International mobility of research, technical and administrative staff at the Charles University. This work was supported by a grant from the Simons Foundation (853830, LR). The third author is also grateful for support from a 2021-2023 Dean's Faculty Fellowship from Vanderbilt University and to the Max Planck Institute for Mathematics in Bonn for its hospitality and financial support. The fourth author has been supported by the Max-Planck-Gesellschaft.\\

\section*{Notation}

For the readers convenience we list the notation that is used in the paper. We always treat $N\in\N_{\ge2}$, $K\in\N_0$ and $1\le\a,\b\le N$, $\a\ne\b$ as fixed parameters.
\begin{itemize}
	\item For $z\in\C$ with $\re(z)>0$, we write $z = \e(1+iy)$ with $\e>0$ and $y\in\R$, and $q = e^{-z} \in \C$.
	
	\item The constant $\l$ is fixed, real, and satisfies $-\frac23<\l<-\frac12$.
	
	\item For $\ell\in\Z$, we write $[\ell]_N$ to denote the smallest positive integer congruent to $\ell\Pmod N$.
	
	\item We define the set $\Sc_{\a,\b}:=\{\bm n=(n_1,\dots,n_N)^T\in\N_0^N:n_\a>n_\b\}$.
	
	\item We write $\bm\ell=(\ell_1,\dots,\ell_N)$ for an element in $(\Z/N\Z)^N$. We always assume $0\le\ell_j<N$.
	
	\item For $\bm\ell=(\ell_1,\dots,\ell_N)\in(\Z/N\Z)^N$, we define $\Sc_{\a,\b;N,\bm\ell}:=\{\bm n\in\Sc_{\a,\b}:\bm n\equiv\bm\ell\Pmod N\}$.
	
	\item We let $\Nc_{\e,\l}:=\{\bm n\in\N_0^N:\bm n=\frac{\log(2)}{\e}\bm1+\bm\mu,|\bm\mu|\le\e^\l\}$ and write $\bm\mu=\bm n-\frac{\log(2)}{\e}\bm1$.
	
	\item For $\bm n\in\N_0^N$, we define $\bm u\in\C^N$ entrywise by $n_j=\frac{\log(2)}{z}+\frac{u_j}{\sqrt z}$, giving a map $\bm n\mapsto\bm u$. The bijective image of $\Sc_{\a,\b;N,\bm\ell}$ (resp. $\Nc_{\e,\l}$) under the map $\bm n\mapsto\bm u$ is denoted $\Tc_{\a,\b;N,\bm\ell}$ (resp. $\Uc_{\e,\l}$). We also define $\bm v\in\C^N$ entrywise by $n_j=\frac{\log(2)}\e+\frac{v_j}{\sqrt z}$.
	
	\item The functions $C_r$, $\La$, and $D_r$, are given as
	\begin{align*}
		\sum_{r\ge0} C_r(\bm u)z^\frac r2 &:= \exp(\phi(\bm u,z)),\qquad\text{with}\quad \phi(\bm u,z) := -\bm b^T\bm u\sqrt{z} - \frac{Nz}{24} + \sum_{j=1}^N \xi\left(\frac{u_j}{\sqrt{z}},z\right),\\
		\text{where}\qquad\xi(\nu,z) &\hspace{.1cm}= -\sum_{r=2}^{R-1} \left(B_r(-\nu)-\d_{r,2}\nu^2\right)\Li_{2-r}\left(\frac12\right)\frac{z^{r-1}}{r!} + O\left(z^{R-1}\right) \quad (R\in\N),\\
		\La(y) &:= N\left(\frac{\pi^2}{6}-\frac{\log(2)^2(1+iy)^2}{2}-\Li_2\left(2^{-(1+iy)}\right)\right),\\
		\sum_{r\ge0} D_r(\bm v,y)z^\frac r2 &:= \exp(\phi(\bm v,z)),\qquad\text{with}\quad \phi(\bm v,z) := -\bm b^T\bm v\sqrt z - \frac{Nz}{24} + \sum_{j=1}^N \xi_y\rb{\frac{v_j}{\sqrt{z}},z},\\
		\text{where}\qquad \xi_y\rb{\nu,z} &\hspace{.1cm}= - \sum_{r=2}^{R-1} \rb{B_r\rb{-\nu} - \delta_{r,2} \nu^2} \Li_{2-r}\rb{2^{-(1+iy)}} \frac{z^{r-1}}{r!} + O\rb{z^{R-1}} \quad (R\in\N).
	\end{align*}
	
	\item The functions $\Dc_{\a,\b;N}^{[K]}$ and $\Dc_{\a,\b;N,\bm\ell}^{[K]}$ are defined as
	\begin{align*}
		\Dc_{\a,\b;N}^{[K]}(q) &:= \sum_{n\ge0} d_{\a,\b;N}^{[K]}(n)q^n = \sum_{\bm n\in\Sc_{\a,\b}} \frac{q^{N\cdot H(\bm n)}}{\prod_{j=1}^N \left(q^N;q^N\right)_{n_j}},\\
		\Dc_{\a,\b;N,\bm\ell}^{[K]}(q) &:= \sum_{\bm n\in\Sc_{\a,\b;N,\bm\ell}} \frac{q^{NH(\bm n)}}{\prod_{j=1}^N \left(q^N;q^N\right)_{n_j}}.
	\end{align*}
	
	\item The functions $f_{\a,\b;N,\bm\ell}^{[K]}$, $g_{\a,\b;N,\bm\ell}^{[K]}$, $g_{\a,\b;N,\bm\ell}^{[K,1]}$, and $g_{\a,\b;N,\bm\ell}^{[K,2]}$ are given as
	\begin{align*}
		f_{\a,\b;N,\bm\ell}^{[K]}(z) &:= \sum_{\bm n\in\Sc_{\a,\b;N,\bm\ell}} \frac{e^{-{NH}(\bm n)z}}{\prod_{j=1}^N \left(e^{-Nz};e^{-Nz}\right)_{n_j}},\qquad g_{\a,\b;N,\bm\ell}^{[K]}(z) := f_{\a,\b;N,\bm\ell}^{[K]}\leg zN,\\
		g_{\a,\b;N,\bm\ell}^{[K,1]}(z) &:= \sum_{\bm n\in\Sc_{\a,\b;N,\bm\ell}\cap\Nc_{\e,\l}} \frac{e^{-H(\bm n)z}}{\prod_{j=1}^N \left(e^{-z};e^{-z}\right)_{n_j}},\quad g_{\a,\b;N,\bm\ell}^{[K,2]}(z) := \sum_{\bm n\in\Sc_{\a,\b;N,\bm\ell}\sm\Nc_{\e,\l}} \frac{e^{-H(\bm n)z}}{\prod_{j=1}^N \left(e^{-z};e^{-z}\right)_{n_j}}.
	\end{align*}
	
	\item For $\bm u\in\C^N$, we let $u_\a$ be the $\a$-th entry of $\bm u$, and $\bm{u_{[1]}}$ be the remaining $N-1$ entries. For convenience, we write $\bm u=(\bm{u_{[1]}},u_\a)$. Analogously, we write $\bm{\mu_{[1]}}$ for the corresponding $N-1$ entries of $\bm\mu$. For $1\le c \le N$, $c\ne\a$, we let $u_c$ (resp. $u_\a$) be the $c$-th (resp. $\a$-th) entry of $\bm u$, and $\bm{u_{[2]}}$ for the remaining $N-2$ entries. For convenience, we write $\bm u=(\bm{u_{[1]}},u_\a)=(\bm{u_{[2]}},u_c,u_\a)$.
	
	\item For fixed $\bm\ell\in(\Z/N\Z)^N$, the sets $\mf u_\a(\bm{u_{[1]}})$, $\mf u_c$, and $\mc U_{[1]}$ are defined as
	\begin{align*}
		\mf u_\alpha\left(\bm{u_{[1]}}\right) &:= \cbd{u_\beta+t\sqrt{z}}{t\in {[\ell_\alpha-\ell_\beta]_N}+N\N_0},\quad \mf u_c := \cbd{-\frac{\log (2)}{\sqrt{z}} + t\sqrt{z}}{t\in\ell_c+N\N_0}{\hlc ,}\\
		\mc U_{\bm{[1]}} &:= \cbd{\bm{u_{[1]}} \in \rb{-\frac{\log(2)}{\sqrt{z}}+\R\sqrt{z}}^{N-1}}{\vb{\bm{\mu_{[1]}}} \le \e^\l}.
	\end{align*}
	
	\item The functions $G_j$ and $G_{j,\a}$ are given by
	\begin{align*}
		G_j(\bm u) := e^{\frac{\pi^2N}{12z}} C_j(\bm u) e^{-\bm u^T\bm u},\qquad G_{j,\a}\left(\bm{u_{[1]}}\right) := \sum_{u_\a\in\mf u_\a\left(\bm{u_{[1]}}\right)} G_j\left(\bm{u_{[1]}},u_\a\right).
	\end{align*}
	
	\item The coefficients $E_{\bm\ell,r}$, $V_{j,r}$ and $W_r$ are defined as, with $R\in\N$,
	\begin{align*}
		g_{\a,\b;N,\bm\ell}^{[K,1]}(z) &=: \frac{e^\frac{\pi^2N}{12z}}{2^{K+\frac12}\pi^\frac N2}\sum_{r=0}^{R-1} E_{\bm\ell,r} z^\frac r2 + e^\frac{\pi^2N}{12\varepsilon} O\rb{\varepsilon^{N\rb{\lambda+\frac 12}+3R\rb{\lambda+\frac 23}}},\\
		\sum_{\bm u\in\Tc_{\a,\b;N,\bm\ell}\cap\Uc_{\e,\l}} G_j(\bm u) &=: e^\frac{\pi^2N}{12z}\sum_{r=-N}^{R-1} V_{j,r}z^\frac r2 + O\left(\e^{\left(\l+\frac12\right)(3j+2N+R-1)+\frac{N+R}{2}-1}e^\frac{\pi^2N}{12\e}\right),\\
		\sum_{u_\a\in\mf u_\a\left(\bm{u_{[1]}}\right)} G_j\left(\bm{u_{[1]}},u_\a\right) &=: \frac{1}{N\sqrt z}\int_{u_\b}^{u_\b+\sqrt z\infty} G_j\left(\bm{u_{[1]}},u_\a\right) du_\a + \sum_{r=0}^{R-1} W_r G_j^{(r)}\left(\bm{u_{[1]}},u_\b\right)z^\frac r2\\
		&\hspace{5.5cm}+ O\left(\e^{\left(\l+\frac12\right)(3j+R+1)+\frac{R-1}{2}}e^\frac{\pi^2N}{12\e}\right).
	\end{align*}
\end{itemize}

\section{Preliminaries}\label{section:prelim}

We first recall several special functions. For $s\in\C$, the {\it polylogarithm} is given by
\ba
	\Li_s(w) &:= \sum\limits_{n\geq1} \frac{w^n}{n^s}, & \vb{w} < 1.
\ea
For $w\in(0,1)$, define the {\it Rogers dilogarithm function} (shifted by a constant so that $L(1)=0$)
\[
	L(w) := \Li_2(w) + \frac12\log(w)\log(1-w) - \frac{\pi^2}{6}.
\]
It is well-known (see \cite[p.6]{Zagier2007}) that
\begin{equation}\label{eq:Rogerdilog}
	L\left(\frac12\right) = -\frac{\pi^2}{12}.
\end{equation}
Let $B_r(x)$ denote the \textit{$r$-th Bernoulli polynomial} defined via the generating function
\ba
	\frac{t e^{xt}}{e^t-1} =: \sum\limits_{n=0}^\infty B_n(x) \frac{t^n}{n!}.
\ea
The \textit{$r$-th Bernoulli number} is given by $B_r := B_r(0)$.

We use $\ll,\gg$, and $\asymp$ to compare the size of quantities, without any assumption on their signs or arguments unless stated otherwise; e.g., for $y\in\R$, $y\ll1$ means $-C\le y\le C$ for some $C>0$.

We first give a basic estimate for $(q;q)_\infty^{-1}$; the proof of the following lemma is straightforward.

\begin{lem}\label{lem:lqqi}
	Let $\e>0$ and $z:=\e(1+iy)$. Suppose $y\ll\e^{-\frac12+\d}$ for some $\d>0$. Then we have\footnote{Throughout, we take the principal branch for the logarithm.}
	\[
		-\Log((q,q)_\infty) = \frac{\pi^2}{6z} + \frac12\Log\left(\frac{z}{2\pi}\right) - \frac{z}{24} + \Ec,
	\]
	as $z\to0$, where the error term $\Ec$ satisfies $\Ec\ll z^L$ for all $L\in\N$.
\end{lem}

The following lemma is a special case of Lemma 2.1 of \cite{GZ2021}.

\begin{lem}\label{lem:GZ}
	Let $z,w\in\C$ with $\re(z)>0$, $|w|<1$, and $\nu\in\C$ such that $\nu z=o(1)$. Then\footnote{In \cite[Lemma 2.1]{GZ2021}, there is a cyclic quantum dilogarithm term, which vanishes if $q$ approaches $1$.}
	\ba
		\Log\rb{\rb{we^{-\nu z} q;q}_\infty} = -\Li_2(w)\frac{1}{z} - \rb{\nu+\frac12} \Log(1-w) - \frac{\nu^2z}{2} \frac{w}{1-w} + \psi_w (\nu,z).
	\ea
	Here $\psi_w(\nu,z)$, for $R\in\N$, has an asymptotic expansion as $z\to0$ with $\re(z)>0$
	\begin{equation}\label{eq:psi_asymp}
		\psi_w(\nu,z) = -\sum_{r=2}^{R-1} \left(B_r(-\nu)-\d_{r,2}\nu^2\right)\Li_{2-r}(w)\frac{z^{r-1}}{r!} + O\left(z^{R-1}\right),
	\end{equation}
	with $\delta_{j,k}$ the Kronecker delta symbol. In particular, for every $n\in\N_0$ the coefficient of $\nu^n$ is $O( z^{\frac{2n}{3}})$. 
\end{lem}

\begin{rmk}
	\cite[Lemma 2.1]{GZ2021} is only stated for real $z$. However, a line-by-line check of the proof therein shows that the statement still holds for complex $z$.
\end{rmk}

We also make use of the following version of the Euler--Maclaurin summation formula. The classical version of Euler--Maclaurin summation compares the sum $\sum_{n=0}^mF(n)$ to the integral $\int_0^mF(x)dx$; this version makes use of the same expressions but with a change of variables $F(x)=f(xz+a)$. A related ``shifted'' version of Euler--Maclaurin, can also be found in Proposition 3 of \cite{ZagierMellin}.

\begin{prp}\label{prp:EMF}
	Let $a,z\in\C$. Let $f$ be a holomorphic function that is defined in $L(a,z):=\{tz+a:t\in\R_0^+\}$. Then we have, for all $m,R\in\N$,
	\begin{multline*}
		\sum_{n=0}^m f(nz+a) = \frac1z\int_a^{a+mz} f(x) dx + \frac{f(mz+a)+f(a)}{2}\\
		+ \sum_{r=1}^R \frac{B_{2r}z^{2r-1}}{(2r)!}\left(f^{(2r-1)}(mz+a)-f^{(2r-1)}(a)\right) + O(1)z^{2R}\int_a^{a+mz} \vb{f^{(2R+1)}(x)} dx.
	\end{multline*}
	Furthermore, if $f$ has rapid decay on $L(a,z)$, then we have
	\begin{multline*}
		\sum_{n\ge0} f(nz+a) = \frac1z\int_a^{a+z\infty} f(x) dx + \frac{f(a)}{2}\\
		- \sum_{r=1}^R \frac{B_{2r}z^{2r-1}}{(2r)!}f^{(2r-1)}(a) + O(1)z^{2R}\int_a^{a+z\infty} \vb{f^{(2R+1)}(x)} dx.
	\end{multline*}
\end{prp}

\section{Asymptotic formulas for summands}\label{section:afs}

\subsection{The setup}

We study the asymptotics as $z\to0$, in the right half-plane, of the function
\begin{equation}\label{eq:gfsd} 
	\frac{e^{-H(\bm n)z}}{\prod_{j=1}^N (q;q)_{n_j}}.
\end{equation}
For $z$ real, asymptotics for \eqref{eq:gfsd} were obtained in \cite{GZ2021,VZ2011}. To avoid confusion, we write $\e>0$ for a real variable and denote by $z=\e(1+iy)$ a complex variable with $\re(z)>0$. For $-\frac 23<\l<-\frac12$, we define, with $|\bm\mu|$ the (Euclidean) norm of the vector $\bm\mu$,
\begin{equation*}
	\Nc_{\e,\l} := \left\{\bm n\in\N_0^N : \bm n=\frac{\log(2)}{\e}\bm1+\bm\mu,\ |\bm\mu|\le\e^\l\right\}.
\end{equation*}

\begin{remark*}
	There is a general philosophy behind considering this expansion. As in \cite{GZ2021} (see page 3), and other references studying Nahm sums, natural vectors to ``expand'' around are solutions to the so-called Nahm equation (essentially, this corresponds to the saddle point). In our situation, $H(\bm n)$ can be decomposed into a quadratic form piece and a linear piece, and the quadratic form piece is simply the form $\frac12|{\bm n}|^2$. In this case, the Nahm equation (see (6) in \cite{GZ2021}) becomes $1-z_j=z_j$, ($1\le j\le N$) which has unique solution in $(0,1)^N$ equal to $\frac12\bm1$. As explained in Proposition 2.1 of \cite{GZ2021}, one should center the expansion around $n_j\approx\frac1\e\log(\frac{1}{z_j})$, which in our case becomes $n_j\approx\frac{\log(2)}{\e}$ for all $j$. Thus, $\frac{\log(2)}{\e}\bm1$ is the saddle point we need to expand around, which, philosophically, explains the definition of $\Nc_{\e,\l}$. This also becomes apparent in the forthcoming proofs.
\end{remark*}

In this section, we prove asymptotic formulas for \eqref{eq:gfsd}, in the case $\bm n\in\Nc_{\e,\l}$.

\subsection{Narrow range estimates}

We prove asymptotic expansions for \eqref{eq:gfsd} if $\bm n\in\Nc_{\e,\l}$.

\begin{prp}\label{prp:senr} 
	Let $\bm n\in\Nc_{\e,\l}$ and $\bm u\in\C^N$ be such that, for $1\le j\le N$,
	\begin{equation}\label{eq:ujd} 
		n_j = \frac{\log (2)}{z} + \frac{u_j}{\sqrt{z}}.
	\end{equation}
	Suppose that $y \ll\varepsilon^\delta$ for some $\delta>0$. Then we have
	\begin{equation}\label{eq:senr} 
		\frac{e^{-H(\bm n)z}}{\prod_{j=1}^N (q;q)_{n_j}} = \rb{\frac{z}{\pi}}^\frac N2 \frac{e^{\frac{\pi^2 N}{12z}-\frac{Nz}{24}}}{2^{K+\frac12}} e^{-\bm u^T\bm u -\bm b^T\bm u\sqrt{z}} \prod\limits_{j=1}^N \exp\rb{\xi\rb{\tfrac{u_j}{\sqrt{z}},z}},
	\end{equation}
	where, for $R\in\N$, $\xi(\nu,z)$ has an asymptotic expansion as $z\to0$
	\begin{equation}\label{eq:xsenr} 
		\xi(\nu,z) = -\sum_{r=2}^{R-1} \left(B_r(-\nu)-\d_{r,2}\nu^2\right)\Li_{2-r}\left(\frac12\right)\frac{z^{r-1}}{r!} + O\left(z^{R-1}\right).
	\end{equation}
\end{prp}

\begin{proof}
	We write
	\begin{equation}\label{eq:qqnq} 
		\frac{1}{(q;q)_{n_j}} = \frac{\left(q^{n_j+1};q\right)_\infty}{(q;q)_\infty}.
	\end{equation}
	By \eqref{eq:ujd} we have $q^{n_j}=\frac12e^{-u_j\sqrt{z}}$. Let $\nu_j:=\frac{u_j}{\sqrt{z}}$. Since $\bm n\in\Nc_{\e,\l}$, we have, for some $|\mu_j|\le\e^\l$,
	\begin{equation}\label{eq:nunr} 
		\nu_j = n_j - \frac{\log (2)}z = \rb{\frac{\log (2)}{\varepsilon} + \mu_j} - \frac{\log (2)}{\varepsilon(1+iy)} = \rb{1-\frac1{1+iy}} \frac{\log (2)}{\varepsilon}+ \mu_j.
	\end{equation}
	If $y\ll\e^\d$, then the first term in \eqref{eq:nunr} has size $\e^{\d-1}$. Since $|\mu_j|\le\e^\l$, we conclude that $\nu_jz=o(1)$. Thus we may apply \Cref{lem:GZ} with $w=\frac12$, and deduce that
	\begin{equation} \label{eq:slqqi} 
		\Log\rb{\rb{q^{n_j+1};q}_\infty} = -\frac1z\Li_2\rb{\frac 12} + \log (2)\rb{\nu_j+\frac12} - \frac{\nu_j^2z}{2} + \psi_{\frac 12}(\nu_j,z).
	\end{equation}
	By \eqref{eq:qqnq}, \eqref{eq:slqqi}, and \Cref{lem:lqqi} we have, with $\mc E$ satisfying $\mc E \ll z^L$ for all $L\in\N$, 
	\begin{equation*}
		-\Log\rb{(q;q)_{n_j}} = -\Li_2\rb{\frac 12}\frac1z + \log (2)\rb{\nu_j+\frac12} - \frac{\nu_j^2z}{2} + \psi_{\frac 12}(\nu_j,z) + \frac{\pi^2}{6z} - \frac 12 \Log\rb{\frac{2\pi}{z}} - \frac{z}{24} + { \mc E}.
	\end{equation*}
	Now let $\xi(\nu_j,z):=\psi_\frac12(\nu_j,z)+\Ec$. Since $\Ec\ll z^L$ for all $L\in\N$, it follows that $\xi(\nu_j,z)$ has the same asymptotic expansion as $\psi_\frac12(\nu_j,z)$. So we may write
	\begin{align*}
		-\Log\rb{(q;q)_{n_j}} &= -\frac12\Log\rb{\frac{\pi}{z}} - \frac{z}{24}+\rb{\frac{\pi^2}{6}-\Li_2\rb{\frac 12}} \frac1z + \log (2)\nu_j - \frac{\nu_j^2z}{2} + \xi(\nu_j,z).
	\end{align*}
	Summing over $j$ gives
	\begin{equation}\label{eq:lqqn} 
		\scalebox{0.89}{$\displaystyle -\sum_{j=1}^N \Log\rb{(q;q)_{n_j}} = -\frac N2\Log\rb{\frac{\pi}{z}}-\frac{Nz}{24} + \rb{\frac{\pi^2}{6}-\Li_2\rb{\frac 12}}\frac Nz +\sum_{j=1}^N \rb{\log (2)\nu_j - \frac{\nu_j^2z}{2} + \xi(\nu_j,z)}.$}
	\end{equation}
	Using \eqref{eq:sbj}, we note that
	\begin{equation}\label{eq:Hnz} 
		-H(\bm n)z = -\frac{N\log (2)^2}{2z} - \frac{\log (2)}{\sqrt{z}}\sum_{j=1}^N u_j - \frac12\bm u^T\bm u- \log(2)\rb{K+\frac 12}- \bm b^T\bm u\sqrt{z}.
	\end{equation}
	We combine \eqref{eq:lqqn} and \eqref{eq:Hnz} and get
	\begin{multline*}
		\hspace{-.2cm}\scalebox{.95}{$\displaystyle\Log\left(\frac{e^{-H(\bm n)z}}{\prod_{j=1}^N (q;q)_{n_j}}\right) = -\frac N2\Log\left(\frac\pi z\right) - \frac{Nz}{24} + \left(\frac{\pi^2}{6}-\Li_2\left(\frac12\right)\right)\frac Nz + \sum_{j=1}^N \left(\log(2)\nu_j-\frac{\nu_j^2z}{2}+\xi(\nu_j,z)\right)$}\\
		\scalebox{.95}{$\displaystyle- \frac{N\log(2)^2}{2z} - \frac{\log(2)}{\sqrt z}\sum_{j=1}^N u_j - \frac12\bm u^T\bm u- \log(2)\left(K+\frac12\right) - \bm b^T\bm u\sqrt z.$}
	\end{multline*}
	Using \eqref{eq:Rogerdilog}, we evaluate
	\ba
		- \rb{\Li_2\rb{\frac 12} + \frac {\log(2)^2}2 - \frac{\pi^2}{6}}\frac Nz = - L\rb{\frac 12}\frac Nz = \frac{\pi^2N}{12z}.
	\ea
	Finally, using this and the fact that $\nu_j = \frac{u_j}{\sqrt{z}}$, we may rewrite
	\begin{multline*}
		\Log\left(\frac{e^{-H(\bm n)z}}{\prod_{j=1}^N (q;q)_{n_j}}\right)\\
		= -\frac N2\Log\left(\frac\pi z\right) + \frac{\pi^2N}{12z} - \frac{Nz}{24} - \bm u^T\bm u - \bm b^T\bm u\sqrt z - \log(2)\left(K+\frac12\right) + \sum_{j=1}^N \xi\left(\frac{u_j}{\sqrt z},z\right).
	\end{multline*}
	Finally, we obtain \eqref{eq:senr} by exponentiating.
\end{proof}

Now we use \Cref{prp:senr} to derive an asymptotic expansion for \eqref{eq:gfsd}.

\begin{prp}\label{prp:sanr} 
	Assume the setup as in \Cref{prp:senr}. If $y\ll\e^{\frac13+\d}$ for some $\d>0$, then we have, for $R\in\N$, uniformly in $\bm u$,
	\begin{equation*}
		\frac{e^{-H(\bm n)z}}{\prod_{j=1}^N (q;q)_{n_j}} = \leg z\pi^\frac N2\frac{e^\frac{\pi^2N}{12z}e^{-\bm u^T\bm u}}{2^{K+\frac12}}\sum_{r=0}^{R-1} C_r(\bm u)z^\frac r2 + O\left(\e^{3R\d_1}\right) z^\frac N2e^\frac{\pi^2N}{12z}e^{-\bm u^T\bm u},
	\end{equation*}
	where $\d_1:=\min\{\d,\frac23+\l\}>0$. The $C_r(\bm u)$ are defined as coefficients of the formal exponential
	\begin{equation}\label{eq:Crud} 
		\sum_{r\ge0} C_r(\bm u)z^\frac r2 := \exp(\phi(\bm u,z)),\qquad \phi(\bm u,z) := -\bm b^T\bm u\sqrt{z} - \frac{Nz}{24} + \sum_{j=1}^N \xi\left(\frac{u_j}{\sqrt{z}},z\right),
	\end{equation}
	and {$\xi(\nu,z)$} has the asymptotic expansion given in \eqref{eq:xsenr}. 
\end{prp}

\begin{proof}
	Using \eqref{eq:senr} and \eqref{eq:Crud}, we have, as a formal power series,
	\[
		\frac{e^{-H(\bm n)z}}{\prod_{j=1}^N (q;q)_{n_j}} = \left(\frac z\pi\right)^\frac N2\frac{e^\frac{\pi^2N}{12z}}{2^{K+\frac12}}e^{-\bm u^T\bm u}\sum_{r\ge0} C_r(\bm u)z^\frac r2.
	\]
	It remains to show that for $R\in\N$, uniformly in $\bm u$,
	\begin{equation}\label{eq:sanr_e}
		\left(\frac z\pi\right)^\frac N2\frac{e^\frac{\pi^2N}{12z}}{2^{K+\frac12}}e^{-\bm u^T\bm u}\sum_{r\ge R} C_r(\bm u)z^\frac r2 = O\left(\e^{3R\d_1}\right)z^\frac N2e^\frac{\pi^2N}{12z}e^{-\bm u^T\bm u}.
	\end{equation}
	
	As $y\ll\e^{\frac13+\d}$, we have $z\asymp\e$. Since $\bm n\in\Nc_{\e,\l}$, we have $|\bm\mu|\le\e^\l$, and it follows from \eqref{eq:nunr} that $\nu_j\ll\e^{-\frac23+\d_1}$ and $|\bm u|\ll\e^{-\frac16+\d_1}$. Now we consider the asymptotic expansion of $\phi(\bm u,z)$. As $\psi_\frac12(\nu_j,z)$ and $\xi(\nu_j,z)$ have the same asymptotic expansion (see the proof of \Cref{prp:senr}), we use \eqref{eq:psi_asymp} and deduce for $S\ge3$ that
	\begin{multline}\label{eq:sanr_p}
		\phi(\bm u,z) = -{{\bm b}}^T\bm u\sqrt{z} - \frac{Nz}{24}
		-\sum_{j=1}^N \sum_{r=2}^{S-1} \rb{B_r\rb{-\frac{u_j}{\sqrt{z}}}-\delta_{r,2}\frac{u_j^2}{z}} \Li_{2-r}\rb{\frac 12} \frac{z^{r-1}}{r!}\\ + O\rb{\varepsilon^{\frac 13 (S-3) + S\delta_1}}.
	\end{multline}
	Uniformity in $\bm u$ follows as the implied constant in the error term is independent of $\bm u$.
	
	Finally, we prove the proposition by showing that exponentiating \eqref{eq:sanr_p} gives a well-defined asymptotic series. For this rewrite \eqref{eq:sanr_p} as a formal power series in $z^\frac12$
	\[
		\phi(\bm u,z) =: \sum_{m\ge1} g_m(\bm u)z^\frac m2.
	\]
	Then, since $\deg(B_m)=m$, it follows that $\deg(g_m)\le m+2$. For $m\in\N$, we have
	\ba
		g_m(\bm u)z^\frac m2 \ll \varepsilon^{(m+2)\rb{-\frac16+\delta_1}+\frac m2} \ll \varepsilon^{\frac{m-1}{3}+(m+2)\delta_1}.
	\ea
	It follows that for all $M\in\N$, we have $\sum_{m=1}^M g_m(\bm u)z^\frac m2=O(\e^{3\d_1})$ uniformly in $\bm u$. We exponentiate \eqref{eq:sanr_p} and obtain a formal expression
	\[
		\exp(\phi(\bm u,z)) = \exp\left(\sum_{m\ge1} g_m(\bm u)z^\frac m2\right) = \prod_{m\ge1} \sum_{k\ge0} \frac{g_m(\bm u)^kz^\frac{mk}{2}}{k!} =: \sum_{r\ge0} C_r(\bm u)z^\frac r2,
	\]
	with $C_0(\bm u)=1$. We claim that $\deg(C_r)\le3r$ for all $r\in\N$. To see this, we observe that
	\[
		\deg(C_r) \le \max_{\sum_\ell m_\ell=r} \left\{\sum_{(m_1,m_2,\dots)} \deg(g_{m_\ell})\right\}.
	\]
	Here the maximum is taken over sequences $(m_1,m_2,\dots)$ of non-negative integers satisfying $\sum_\ell m_\ell=r$. But since $\deg(g_m)\le m+2\le3m$ for $m\in\N$, we deduce that $\deg(C_r)\le3r$ for all $r\in\N$ as claimed. From $|\bm u|\ll\e^{-\frac16+\d_1}$ it follows that for all $R\in\N$, uniformly in $\bm u$,
	\[
		\exp(\phi(\bm u,z)) = \sum_{r=0}^{R-1} C_r(\bm u)z^\frac r2 + O\left(\e^{3R\d_1}\right).
	\]
	This proves \eqref{eq:sanr_e}, and hence the proposition.
\end{proof}

\subsection{Wider range estimates}

To establish an alternative asymptotic expansion for \eqref{eq:gfsd}, which is valid on the major arc, let $\Ac:=(1+\frac{2^{-(1+iy)}}{1-2^{-(1+iy)}})I_N$, with $I_N$ the $N\times N$ identity matrix, and
\begin{equation*}
	\La(y) := N\left(\frac{\pi^2}{6}-\frac{\log(2)^2(1+iy)^2}{2}-\Li_2\left(2^{-(1+iy)}\right)\right).
\end{equation*}

\begin{prp}\label{prp:sewr}
	Let $\bm n\in\mc N_{\varepsilon,\l}$ and $\bm v\in\C^N$ be such that for $1\le j\le N$, 
	\begin{equation}\label{eq:vjd} 
		n_j = \frac{\log (2)}{\varepsilon} + \frac{v_j}{\sqrt{z}}.
	\end{equation}
	Suppose that $y \ll\varepsilon^{-1-\frac{3\lambda}2+\delta}$ for some $\delta>0$. Then we have
	\begin{multline*}
		\frac{e^{-H(\bm n)z}}{\prod_{j=1}^N (q;q)_{n_j}} = 2^{-\left(K+\frac12\right)(1+iy)}\left(1-2^{-(1+iy)}\right)^{-\frac N2-\frac{1}{\sqrt z}{\sum_{j=1}^N v_j}}\left(\frac{z}{2\pi}\right)^\frac N2\\
		\times \exp\left(\frac{\La(y)}{z}-\frac{Nz}{24}-\frac{\log(2)(1+iy)}{\sqrt z}\sum_{j=1}^N v_j - \frac12\bm v^T\Ac\bm v-\bm b^T\bm v\sqrt z\right)\prod_{j=1}^N \exp\left(\xi_y\left(\frac{v_j}{\sqrt z},z\right)\right),
	\end{multline*}
	where, for $R\in\N$, $\xi_y(\nu,z)$ has an asymptotic expansion, as $z\to0$,
	\begin{equation}\label{eq:xsewr} 
		\xi_y\rb{\nu,z} = - \sum_{r=2}^{R-1} \rb{B_r\rb{-\nu} - \delta_{r,2} \nu^2} \Li_{2-r}\rb{2^{-(1+iy)}} \frac{z^{r-1}}{r!} + O\rb{z^{R-1}}.
	\end{equation}
\end{prp}

\begin{rmk}
	Note the difference between $\bm u$ and $\bm v$ in \eqref{eq:ujd} and \eqref{eq:vjd}. Writing $\bm n$ in terms of $\bm v$ (instead of $\bm u$) allows to derive an asymptotic expansion which holds for a wider range in $y$.
\end{rmk}

\begin{proof}[Proof of \Cref{prp:sewr}]
	By \eqref{eq:vjd}, we have $q^{n_j}=2^{-(1+iy)}e^{-v_j\sqrt{z}}$. Define $\nu_j:=\frac{v_j}{\sqrt{z}}$. Then we have
	\begin{equation*}
		\nu_j = n_j - \frac{\log (2)}{\varepsilon} = O\rb{\varepsilon^\lambda}.
	\end{equation*}
	If $y\ll\e^{-1-\frac{3\l}{2}+\d}$, then $z\ll\e^{-\frac{3\l}2+\d_2}$, where $\d_2:=\min\{1+\frac{3\l}{2},\d\}>0$ and $\nu_jz=o(1)$.
	
	Define $\omega:= 2^{-(1+iy)}$. Observing that $|\omega|=|2^{-(1+iy)}|<1$, we may apply \Cref{lem:GZ} and deduce 
	\[
		\Log\left(\left(q^{n_j+1};q\right)_\infty\right) = -\frac{\Li_2(\om) }z- \left(\nu_j+\frac12\right)\Log(1-\om) - \frac{\nu_j^2z}{2}\frac{\om}{1-\om} + \psi_\om(\nu_j,z).
	\]
	Using \Cref{lem:lqqi} and \eqref{eq:qqnq}, we sum over $j$ and get, with $\Ec$ satisfying $\Ec\ll z^L$ for all $L\in\N$,
	\begin{multline*}
		-\sum_{j=1}^N \Log((q;q)_{n_j}) = -\frac {N\Li_2(\om)}z + \sum_{j=1}^N \left(- \left(\nu_j+\frac12\right)\Log(1-\om) - \frac{\nu_j^2z}{2}\frac{\om}{1-\om} + \psi_\om(\nu_j,z)\right)\\
		+\frac{\pi^2N}{6z} -\frac N2\Log\left(\frac{2\pi}{z}\right) - \frac{Nz}{24} + N\Ec.
	\end{multline*}
	Now let $\xi_y(\nu_j,z):=\psi_\om(\nu_j,z)+\Ec$. As we have $\Ec\ll z^L$ for all $L\in\N$, it follows that $\xi_y(\nu_j,z)$ has the same asymptotic expansion as $\psi_\om(\nu_j,z)$. Next, using \eqref{eq:sbj}, we have
	\[
		-H(\bm n)z = -\frac{N\log(2)^2(1+iy)}{2\e} - \frac{\log(2)(1+iy)}{\sqrt z}\sum_{j=1}^N v_j - \frac12\bm v^T\bm v - \log(2)\left(K+\frac12\right)(1+iy) - \bm b^T\bm v\sqrt z.
	\]
	The rest of the computation is analogous to that of \Cref{prp:senr}.
\end{proof}

Now we use \Cref{prp:sewr} to derive another asymptotic expansion for \eqref{eq:gfsd}; the proof of the proposition follows mutatis mutandis the proof of \Cref{prp:sanr} and is omitted here.

\begin{prp}\label{prp:sawr}
	Assume the setup in \Cref{prp:sewr}. If $y\ll\e^{-1-\frac{3\l}{2}+\d}$ for some $\d>0$, then we have an asymptotic expansion for $R\in\N$, uniformly in $\bm v$,
	\begin{multline*}
		\frac{e^{-H(\bm n)z}}{\prod_{j=1}^N (q;q)_{n_j}} = 2^{-\rb{K+\frac 12}(1+iy)} \left(1-2^{-(1+iy)}\right)^{-\frac N2}\left(\frac{z}{2\pi}\right)^\frac N2e^{\frac{\La(y)}{z}}\\
		\times e^{-\frac12\bm v^T\Ac\bm v+\frac{1}{\sqrt z}\left(-\log(2)(1+iy)-\Log\left(1-2^{-(1+iy)}\right)\right)\sum_{j=1}^N v_j}\sum_{r=0}^{R-1} D_r(\bm v,y)z^\frac r2\\
		+ z^\frac N2e^\frac{\La(y)}{z}e^{-\frac12\bm v^T\Ac\bm v+\frac{1}{\sqrt z}\left(-\log(2)(1+iy)-\Log\left(1-2^{-(1+iy)}\right)\right)\sum_{j=1}^N v_j}O\left(\e^{2R\d_2}\right),
	\end{multline*}
	where $\d_2:=\min\{1+\frac{3\l}{2},\d\}>0$, $D_r(\bm v,y)$ are defined as coefficients of the formal exponential
	\begin{equation*}
		\sum_{r\geq0} D_r(\bm v,y)z^\frac r2 := \exp(\phi(\bm v,z)),\qquad \phi(\bm v,z) := -\bm b^T\bm v\sqrt{z} - \frac{Nz}{24} + \sum_{j=1}^N \xi_y\rb{\frac{v_j}{\sqrt{z}},z},
	\end{equation*}
	and $\xi_y(\frac{v_j}{\sqrt{z}},z)$ has the asymptotic expansion given in \eqref{eq:xsewr}.
\end{prp}
\section{Error Estimates}\label{section:error}

In this section, we establish some error estimates which are used in \Cref{section:afp}.

\begin{prp}\label{prp:ore} 
For every $L\in\N$, as $\re(z)\to0$ in the right half-plane, we have
\[
\sum_{\bm n\in\N_0^N\setminus\Nc_{\e,\l}} \vb{\frac{e^{-H(\bm n)z}}{\prod_{j=1}^N (q;q)_{n_j}}} \ll \e^L e^\frac{\pi^2N}{12\e}.
\]
\end{prp}

\begin{proof}
	First suppose that $z=\e$ is real, i.e., $y=0$. Fix $\bm n\in\N_0^N\sm\Nc_{\e,\l}$, and let $\bm\mu\in\R^N$ be such that $n_j=\frac{\log(2)}{\e}+\mu_j$ for $1\le j\le N$. We use \eqref{eq:qqnq}. Observe that $q^{n_j}=\frac12e^{-\mu_j\e}$. Using the power series expansion $\log(1-x)=-\sum_{n\ge1}\frac{x^n}{n}$, we expand
	\begin{align}\label{eq:oors_le} 
		\sum_{j=1}^N \log\rb{\rb{q^{n_j+1};q}_\infty} = -\sum_{j=1}^{N}\sum_{k\geq1} \frac{1}{k\cdot2^{k}} \frac{e^{-k\mu_j\varepsilon}}{e^{k\varepsilon}-1}.
	\end{align}
	Since $e^x>1+x$ for all $x\ne0$ and $\frac{1}{e^x-1}>\frac1x-\frac12$ for $x>0$, we have
	\ba
		\frac{e^{-k\mu_j\varepsilon}}{e^{k\varepsilon}-1} > (1-k\mu_j\varepsilon) \rb{\frac{1}{k\varepsilon}-\frac 12} = \frac{1}{k\varepsilon} - \rb{\mu_j+ \frac 12} + \frac{k\mu_j\varepsilon}{2}.
	\ea
	Plugging this into \eqref{eq:oors_le} gives
	\begin{equation}\label{eq:oors_lb} 
		\sum_{j=1}^N \log\left(\left(q^{n_j+1};q\right)_\infty\right) < -N\Li_2\left(\frac12\right)\frac1\e + \sum_{j=1}^N \left(\log(2)\left(\mu_j+\frac12\right)-\frac{\mu_j\e}{2}\right).
	\end{equation}
	Combining \eqref{eq:oors_lb} and \Cref{lem:lqqi} yields for $L\in\N$, as $\e\to0$,
	\begin{equation}\label{eq:oors_lq} 
		\log\rb{\tfrac{e^{-H(\bm n)\varepsilon}}{\prod_{j=1}^N (q;q)_{n_j}}} < -H(\bm n)\varepsilon + \tfrac{\pi^2N}{6\varepsilon} - \Li_2\rb{\tfrac 12}\tfrac{N}{\varepsilon} - \tfrac N2\log\rb{\tfrac{\pi}{\varepsilon}} - \tfrac{N\varepsilon}{24} + \rb{\log (2) - \tfrac{\varepsilon}{2}}\sum_{j=1}^N \mu_j + N\mc E.
	\end{equation}
	
	Next we compute
	\begin{equation}\label{eq:oors_Hne} 
		-H(\bm n)\e = -\frac{\log(2)^2N}{2\e} - \log(2)\sum_{j=1}^N \mu_j - \bm\mu^T\bm\mu\frac\e2 - \log(2)\sum_{j=1}^N b_j - \bm b^T\bm\mu\e.
	\end{equation}
	Plugging \eqref{eq:oors_Hne} into \eqref{eq:oors_lq} and using \eqref{eq:Rogerdilog} gives 
	\begin{align}\label{eq:oors_T} 
		\log\rb{\frac{e^{-H(\bm n)\varepsilon}}{\prod_{j=1}^N (q;q)_{n_j}}} < &\; \frac{\pi^2N}{12\varepsilon} + T,
	\end{align}
	where 
	\begin{equation*}
		T := - \frac N2 \log\rb{\frac{\pi}{\varepsilon}} - \frac{N\varepsilon}{24} - \frac{\varepsilon}{2}\sum_{j=1}^N \mu_j - \bm\mu^T \bm\mu \frac{\varepsilon}{2} -\log (2)\sum_{j=1}^N b_j - \bm b^T \bm \mu \varepsilon + N\mc E.
	\end{equation*}
	If $\d>0$, $l<-\frac12-\d$, and $|\bm\mu|>\e^l$, then $T$ is dominated by $-\bm\mu^T\bm\mu\frac\e2$ and there exists $\e_0>0$, independent of $l$, such that $T<-\frac14\e^{2l+1}$ for $\e<\e_0$. Plugging this into \eqref{eq:oors_T} shows, for $\e<\e_0$, that
	\begin{equation}\label{eq:oors_qb} 
		\frac{e^{-H(\bm n)\e}}{\prod_{j=1}^N (q;q)_{n_j}} < e^{\frac{\pi^2N}{12\e}- \frac{\e^{2l+1}}4}.
	\end{equation}

	Now consider the sum
	\begin{equation}\label{E:sumr}
		\sum_{\bm n\in\N_0^N\setminus\Nc_{\e,\l}} \frac{e^{-H(\bm n)\e}}{\prod_{j=1}^N (q;q)_{n_j}} = \sum_{r\ge1} \sum_{\substack{\bm n\in\N_0^N\setminus\Nc_{\e,\l}\\\e^{-r+1}<|\bm\mu|\le\e^{-r}}} \frac{e^{-H(\bm n)\e}}{\prod_{j=1}^N (q;q)_{n_j}} =: \sum_{r\ge1} \Rc(r).
	\end{equation}
	For $\Rc(1)$, the sum is over $\bm n\in\N_0^N\sm\Nc_{\e,\l}$ with $1<|\bm\mu|\le\e^{-1}$. Note that $\bm n\in\N_0^N\sm\Nc_{\e,\l}$ implies $|\bm\mu|>\e^\l$, so we are really summing the terms with $\e^\l<|\bm\mu|\le\e^{-1}$. From \eqref{eq:oors_qb}, the summands in $\Rc(1)$ are bounded by $e^{\frac{\pi^2N}{12\e}-\frac{\e^{2\l+1}}{4}}$. Since $\Rc(1)$ contains $O(\e^{-N})$ terms, we conclude that for $\e<\e_0$
	\begin{equation}\label{eq:oors_R1b} 
		\Rc(1) \ll \frac{e^{\frac{\pi^2N}{12\e}-\frac{\e^{2\l+1}}4}}{\e^N}.
	\end{equation}
	Now consider the summands in $\Rc(r)$ for $r\ge2$. Since $|\bm\mu|>\e^{-r+1}$, it follows from \eqref{eq:oors_qb} that the summand is bounded by $e^{\frac{\pi^2N}{12\e}-\frac{\e^{3-2r}}{4}}$. As $\Rc(r)$ contains $O(\e^{-Nr})$ terms, we conclude that for $\e<\e_0$,
	\begin{equation}\label{eq:oors_Rrb} 
		\Rc(r) \ll \frac{e^{\frac{\pi^2N}{12\e}-\frac{\e^{3-2r}}{4}}}{\e^{Nr}}.
	\end{equation}
	Plugging \eqref{eq:oors_R1b} and \eqref{eq:oors_Rrb} into \eqref{E:sumr}, we deduce that
	\[
		\sum_{\bm n\in\N_0^N\sm\Nc_{\e,\l}} \frac{e^{-H(\bm n)\e}}{\prod_{j=1}^N (q;q)_{n_j}} = \sum_{r\ge1} \Rc(r) \ll e^\frac{\pi^2N}{12\e}\left(\frac{e^{-\e^\frac{2\l+1}{4}}}{\e^N}+\sum_{r\ge2} \frac{e^{-\frac{\e^{3-2r}}{4}}}{\e^{Nr}}\right) \ll \e^Le^\frac{\pi^2N}{12\e}
	\]
	for all $L\in\N$ as $\varepsilon\to 0$. This establishes the statement for $z$ real.

	The general case follows from the trivial bound
	\begin{equation*}
		\vb{\frac{e^{-H(\bm n)z}}{\prod_{j=1}^N (q;q)_{n_j}}} \leq \frac{e^{-H(\bm n)\varepsilon}}{\prod_{j=1}^N \rb{\vb{q};\vb{q}}_{n_j}}. \qedhere
	\end{equation*}
\end{proof}

The following proposition helps to establish the minor arc estimate for \eqref{eq:ggf}.

\begin{prp}\label{prp:mae} 
Let $\bm n\in\Nc_{\e,\l}$, and suppose that $\e^{-\d}<|y|\le\frac\pi\e$ for some $\d>0$ sufficiently small. Then we have, for all $L\in\N$,
\ba
\frac{e^{-H(\bm n)z}}{\prod_{j=1}^N (q;q)_{n_j}} \ll \varepsilon^L{e^\frac{\pi^2 N}{12\varepsilon}}.
\ea
\end{prp}

\begin{proof}
	As $q=e^{-z}=e^{-\e(1+iy)}$, the assumption $\e^{-\d}<|y|\le\frac\pi\e$ implies that $\e^{1-\d}<|\Arg(q)|\le\pi$. For $1\le j\le N$, define the sets 
	\[
		\mc Q_j := \cb{q^k : 1\le k\le n_j},\qquad \mc Q_j^- := \cb{s\in\mc Q_j : \re(s)\le0},\qquad \mc Q_j^+ := \mc Q_j\bs\mc Q_{j}^-.
	\]
	If $\e^{1-\d}<|\Arg(q)|\le\pi$ and $\bm n\in\Nc_{\e,\l}$, then one can show that there exists $C>0$, independent of $j$, such that $|\Qc_j^-|\ge C n_j$ for $\e>0$ sufficiently small.

	Next we give an upper bound for $|(q;q)_{n_j}^{-1}|$. For this, we note that for $s\in\C$, we have $|1-s|>1$ if $\re(s)<0$. It follows that
	\[
		\vb{\frac{1}{(q;q)_{n_j}}} \le \prod_{\substack{1\le k\le n_j\\q^k\in\Qc_j^+}} \frac{1}{\left|\rule{0pt}{3.5mm}{1-q^k}\right|} \le \prod_{\substack{1\le k\le n_j\\q^k\in\Qc_j^+}} \frac{1}{1-|q|^k}.
	\]
	As shown above for $\e>0$ sufficiently small, we have $|\Qc_j^-|\ge Cn_j$. It follows that the number of terms in the product is bounded above by $|\Qc_j^+|\le(1-C)n_j$. As $(1-|q|^k)^{-1}$ decreases as $k$ increases, we obtain an upper bound
	\[
		\vb{\frac{1}{(q;q)_{n_j}}} \le \prod_{\substack{1\le k\le n_j\\q^k\in\Qc_j^+}} \frac{1}{1-|q|^k} \le \prod_{k=1}^{\left\lfloor{(1-C)n_j}\right\rfloor} \frac{1}{1-|q|^k} = \frac{1}{(|q|;|q|)_{\left\lfloor{(1-C)n_j}\right\rfloor}}.
	\]
	This yields
	\[
		\vb{\frac{e^{-H(\bm n)z}}{\prod_{j=1}^N (q;q)_{n_j}}} \le \frac{e^{-H(\bm n)\varepsilon}}{\prod_{j=1}^N (|q|;|q|)_{\left\lfloor (1-C)n_j\right\rfloor}}.
	\]
	
	Finally, if $\bm n\in\Nc_{\e,\l}$ and $\e>0$ is sufficiently small, then we have
	\[
		\frac{e^{-H(\bm n)\e}}{\prod_{j=1}^N (|q|;|q|)_{\flo{(1-C)n_j}}} \le \frac{e^{-H(\flo{(1-C)n_1},\dots,\flo{(1-C)n_N})\e}}{\prod_{j=1}^N (|q|;|q|)_{\flo{(1-C)n_j}}} = o\left(\e^Le^\frac{\pi^2N}{12\e}\right)
	\]
	for all $L\in\N$ by \Cref{prp:ore}, since $(\flo{(1-C)n_1},\dots,\flo{(1-C)n_N})\not\in\Nc_{\e,\l}$.
\end{proof}

\section{Asymptotics for $d_{\alpha,\beta;N}^{[K]}(n)$}\label{section:afp}

We split the sum $\Dc_{\a,\b;N}^{[K]}$ as
\begin{equation}\label{eq:gfDd} 
	\mc D_{\alpha,\beta;N}^{[K]}(q) = \sum_{\bm\ell\in(\Z/N\Z)^N} \mc D_{\alpha,\beta;N,\bm\ell}^{[K]}(q),\qquad \mc D_{\alpha,\beta;N,\bm\ell}^{[K]}(q) := \sum_{\bm n\in\mc S_{\alpha,\beta;N,\bm\ell}} \frac{q^{NH(\bm n)}}{\prod_{j=1}^N \rb{q^N;q^N}_{n_j}},
\end{equation}
where $\Sc_{\a,\b;N,\bm\ell}:=\{\bm n\in\Sc_{\a,\b}:\bm n\equiv\bm\ell\Pmod N\}$. We always pick the representative for $\bm\ell\in(\Z/N\Z)^N$ with $0\le\ell_j<N$ ($1\le j\le N$). Let $\z_N := e^{\frac{2\pi i}N}$. Then we have
\begin{equation*}
	\mc D_{\alpha,\beta;N,\bm\ell}^{[K]}(\zeta_N e^{-z}) = \zeta^{NH(\bm\ell)}_Nf_{\alpha,\beta;N,\bm\ell}^{[K]}(z), 
\end{equation*}
where
\begin{equation*}
	f_{\alpha,\beta;N,\bm\ell}^{[K]}(z) := \sum_{\bm n\in\mc S_{\alpha,\beta;N,\bm\ell}} \frac{e^{-{NH}(\bm n)z}}{\prod_{j=1}^N \rb{e^{-Nz};e^{-Nz}}_{n_j}}.
\end{equation*}
It is more convenient to investigate
\begin{equation*}
	g_{\alpha,\beta;N,\bm\ell}^{[K]}(z) := f_{\alpha,\beta;N,\bm\ell}^{[K]}\rb{\frac zN} = \sum_{\bm n\in\mc S_{\alpha,\beta;N,\bm\ell}} \frac{e^{-H(\bm n)z}}{\prod_{j=1}^N \rb{e^{-z};e^{-z}}_{n_j}}.
\end{equation*}
We split
\begin{equation}\label{eq:gKs} 
	g_{\a,\b;N,\bm\ell}^{[K]}(z) = g_{\a,\b;N,\bm\ell}^{[K,1]}(z) + g_{\a,\b;N,\bm\ell}^{[K,2]}(z),
\end{equation}
where
\begin{equation*}
	g_{\alpha, \beta;N,\bm\ell}^{[K,1]}(z) := \sum\limits_{\bm n \in \mc S_{\alpha, \beta;N,\bm\ell} \cap \mc N_{\varepsilon,\lambda}} \frac{e^{-H(\bm n)z}}{\prod_{j=1}^N (e^{-z};e^{-z})_{n_j}}, \quad g_{\alpha, \beta;N,\bm\ell}^{[K,2]}(z) := \sum\limits_{\bm n \in \mc S_{\alpha, \beta;N,\bm\ell} \bs \mc N_{\varepsilon,\lambda}} \frac{e^{-H(\bm n)z}}{\prod_{j=1}^N (e^{-z};e^{-z})_{n_j}}.
\end{equation*}

\begin{prp}\label{prp:gK1a} 
	If $y \ll\varepsilon^{1+\lambda+\delta}$ for some $\delta>0$, then we have, for $R\in\N$ as $z\to 0$,
	\begin{equation*}
		g_{\a,\b;N,\bm\ell}^{[K,1]}(z) = \frac{e^\frac{\pi^2N}{12z}}{2^{K+\frac12}\pi^\frac N2}\sum_{r=0}^{R-1} E_{\bm\ell,r} z^\frac r2 + O\rb{\varepsilon^{N\rb{\lambda+\frac 12}+3R\rb{\lambda+\frac 23}}e^\frac{\pi^2N}{12\varepsilon}},
	\end{equation*}
	where the $E_{\bm\ell,r}$ are explicitly computable. In particular, for every $L\in\N$, we have as, $z\to 0$, with $R_0=R_0(L):=\cei{\frac{L-N(\l+\frac12)}{3\l+2}}$
	\[
		g_{\a,\b;N,\bm\ell}^{[K,1]}(z) = \frac{e^\frac{\pi^2N}{12z}}{2^{K+\frac12}\pi^\frac N2}\sum_{r=0}^{R_0-1} E_{\bm\ell,r}z^\frac r2 + O\left(\e^Le^\frac{\pi^2N}{12\e}\right).
	\]
\end{prp}

We prove \Cref{prp:gK1a} in a series of lemmas. Let $\Tc_{\a,\b;N,\bm\ell}$ be the bijective image of $\Sc_{\a,\b;N,\bm\ell}$ under $\bm n\mapsto\bm u$ given in \eqref{eq:ujd}, and $\Uc_{\e,\l}$ the bijective image of $\Nc_{\e,\l}$ under the same map. Since $\bm n\mapsto\bm u$ takes $\R^N$ to $(-\frac{\log(2)}{\sqrt z}+\R\sqrt z)^N$, we have $\Tc_{\a,\b;N,\bm\ell}\sbe(-\frac{\log(2)}{\sqrt z}+\R\sqrt z)^N\sbe\C^N$.

\begin{lem}\label{lem:uoor} 
	If $y\ll \varepsilon^{1+\lambda+\delta}$ for some $\delta>0$ and $P(\bm u)$ is a polynomial in $\bm u$, then, as $z\to 0$,
	\[
		\sum_{\bm u\in\Tc_{\a,\b;N,\bm\ell}\sm\Uc_{\e,\l}} \vb{P(\bm u)e^{-\bm u^T\bm u}} = O\left(\e^L\right), \quad\text{and}\quad \int_{\left(-\frac{\log(2)}{\sqrt z}+\R\sqrt z\right)^N\big\bs\Uc_{\e,\l}} \vb{P(\bm u)e^{-\bm u^T\bm u}} \bm{du} = O\left(\e^L\right).
	\]
	for all $L\in\N$.
\end{lem}

\begin{proof}
	Let $\bm\mu \in \R^N$ be such that $n_j = \frac{\log(2)}{\varepsilon}+\mu_j$ for $1\le j \le N$. We may rewrite
	\begin{equation}\label{eq:uje} 
		u_j = \frac{\log(2)\sqrt{z}}{\varepsilon} \rb{1-\frac{1}{1+iy}} + \mu_j \sqrt{z}. 
	\end{equation}
	For $y\ll \varepsilon^{1+\lambda+\delta}$, the first term in \eqref{eq:uje} has size $\ll\varepsilon^{\frac 12+\lambda+\delta}$. This implies that
	\[
		\im(u_j) \ll \varepsilon^{\frac 12+\lambda+\delta} + \varepsilon^{\frac 32+\lambda+\delta}\mu_j.
	\]
	Meanwhile, if $\mu_j\gg\e^\l$, then we have $\re(u_j)\gg\e^\frac12\mu_j$. In particular, this implies that $\im(u_j)\ll\e^\d\re(u_j)$. As $\bm n\notin\Nc_{\e,\l}$, we have $|\bm\mu|>\e^\l$, and in particular there is some $j$ such that $\mu_j \gg\e^\l$ for some $j$. Thus $-\bm u^T\bm u$ has negative real part of size $\gg\e^{1+2\l}$ {as $\varepsilon \to 0$}. As $1+2\l<0$, this implies 
	\[
		\sum_{\bm u\in\Tc_{\a,\b;N,\bm\ell}\bs\Uc_{\e,\l}} \vb{P(\bm u)e^{-\bm u^T\bm u}} = O\left(\e^L\right), \quad\text{and}\quad \int_{\left(-\frac{\log(2)}{\sqrt z}+\R\sqrt z\right)^N\big\bs\Uc_{\e,\l}} \vb{P(\bm u)e^{-\bm u^T\bm u}}\bm{du} = O\left(\e^L\right)
	\]
	for all $L\in\N$, noting that $P(\bm u)$ only has polynomial growth.
\end{proof}

Now we prove a bound on the size of the exponential factors appearing in \Cref{prp:sanr}.
\begin{lem}\label{lem:eb} 
	If $y\ll\e^{1+\l+\d}$ for some $\d>0$, and $\bm u\in\Uc_{\e,\l}$, then we have, as $z\to0$,
	\[
		\vb{e^\frac{\pi^2N}{12z}e^{-\bm u^T\bm u}} \le e^\frac{\pi^2N}{12\e}.
	\]
\end{lem}

\begin{proof}
	We first estimate the real and the imaginary parts of $\sqrt z$. Writing $\sqrt z:=\e_0(1+iy_0)$, we have $\e=\e_0^2(1-y_0^2)$ and $y=\frac{2y_0}{1-y_0^2}$. Since $\e>0$, we have $1-y_0^2>0$, thus $|y_0|<1$. So $0<1-y_0^2<1$, and $y=\frac{2y_0}{1-y_0^2}\ge2y_0$. Since $y\ll\e^{1+\l+\d}$, we have $1-y_0^2=1+O(\e^{2+2\l+2\d})$, and $\e_0=\e^\frac12(1+O(\e^{2+2\l+2\d}))$.
	
	We next bound $e^{\re(-\bm u^T\bm u)}$. For this, we compute
	\begin{equation}\label{eq:real}
		\re\left(-\bm u^T\bm u\right) = \sum_{j=1}^N \re\left(-u_j^2\right) = \sum_{j=1}^N \left(\im(u_j)^2-\re(u_j)^2\right).
	\end{equation}
	Since
	\[
		u_j = -\frac{\log(2)}{\e_0}\left(1-iy_0-y_0^2+O\left(y_0^3\right)\right) + n_j\e_0(1+iy_0),
	\]
	it follows that
	\begin{align}\nonumber
		\im(u_j) &= \frac{\log(2)}{\e_0}\left(y_0+O\left(y_0^3\right)\right) + n_j\e_0y_0,\\
		\label{eq:imuj2e}
		\im(u_j)^2 &= \frac{\log(2)^2}{\e_0^2}\left(y_0^2+O\left(y_0^3\right)\right) + n_j\log(2)\left(y_0^2+O\left(y_0^3\right)\right) + n_j^2\e_0^2y_0^2.
	\end{align}
	Since $\bm u\in\Uc_{\e,\l}$, we deduce that $n_j=\frac1\e(\log(2)+O(\e^{1+\l}))$. Plugging this into \eqref{eq:imuj2e}, we get
	\begin{align*}
		\im(u_j)^2 &= \frac{\log(2)^2}{\e_0^2}\left(y_0^2+O\left(y_0^3\right)\right) + \frac1\e\left(\log(2)+O\left(\e^{1+\l}\right)\right)\log(2)\left(y_0^2+O\left(y_0^3\right)\right)\\
		&\hspace{8cm}+ \frac{1}{\e^2}\left(\log(2)+O\left(\e^{1+\l}\right)\right)^2\e_0^2y_0^2.
	\end{align*}
	Using $\e=\e_0^2(1-y_0^2)$, $(1-y_0^2)^{-1}=1+y_0^2+O(y_0^4)$, and $y_0 \ll \e^{1+\l+\d}$, we may expand
	\begin{align*}
		\im(u_j)^2 &= \frac{\log(2)^2\left(1-y_0^2\right)}{\e}\left(y_0^2+O\left(y_0^3\right)\right) + \frac1\e\left(\log(2)+O\left(\e^{1+\l}\right)\right)\log(2)\left(y_0^2+O\left(y_0^3\right)\right)\\
		&\hspace{6cm} +\frac1\e\left(\log(2)+O\left(\e^{1+\l}\right)\right)^2y_0^2\left(1+y_0^2+O\left(y_0^4\right)\right)\\
		&= \frac{\log(2)^2y_0^2}{\e}(1+O(y_0)) + \frac{\log(2)^2y_0^2}{\e}\left(1+O\left(\e^{1+\l}\right)\right) + \frac{\log(2)^2y_0^2}{\e}\left(1+O\left(\e^{1+\l}\right)\right)\\
		&= \frac{y_0^2}{\e}\left(3\log(2)^2+O\left(\e^{1+\l}\right)\right) \le \frac{y^2}{\e}\left(\frac34\log(2)^2+O\left(\e^{1+\l}\right)\right).
	\end{align*}
It follows by \eqref{eq:real} that
	\begin{equation}\label{eq:real2}
		\re\left(-\bm u^T\bm u\right) \le \frac{Ny^2}{\e}\left(\frac34\log(2)^2+O\left(\e^{1+\l}\right)\right).
	\end{equation}

On the other hand, we have
	\[
		\re\left(\frac{\pi^2N}{12z}\right) = \frac{\pi^2N}{12\e}\re\left(\frac{1}{1+iy}\right) = \frac{\pi^2N}{12\e}\left(1-y^2+O\left(y^4\right)\right).
	\]
	Combining this with \eqref{eq:real2} we obtain, for $\e>0$ sufficiently small
	\[
		\re\left(\frac{\pi^2N}{12z}\right) + \re\left(-\bm u^T\bm u\right) = \frac N\e\left(\frac{\pi^2}{12} + y^2\left(\frac34\log(2)^2 - \frac{\pi^2}{12} + O\left(\e^{1+\l}\right)\right) + O\left(y^4\right)\right) \le \frac{\pi^2N}{12\e}. \qedhere
	\]
\end{proof}

We next show the following lemma.

\begin{lem}\label{lem:gK1sa}
	If $y\ll\e^{1+\l+\d}$ for some $\d>0$, then for $R\in\N$ we have, as $z\to 0$,
	\[
		g_{\a,\b;N,\bm\ell}^{[K,1]}(z) = \left(\frac z\pi\right)^\frac N2\frac{e^\frac{\pi^2N}{12z}}{2^{K+\frac12}}\sum_{r=0}^{R-1} \sum_{\bm u\in\Tc_{\a,\b;N,\bm\ell}\cap\Uc_{\e,\l}} e^{-\bm u^T\bm u} C_r(\bm u)z^\frac r2 + O\rb{\varepsilon^{N\rb{\lambda+\frac 12}+3R\rb{\lambda+\frac 23}}e^\frac{\pi^2N}{12\varepsilon}}.
	\]
\end{lem}

\begin{proof}
	Summing the asymptotics in \Cref{prp:sanr}, we obtain\footnote{Note that $\delta_1$ in \Cref{prp:sanr} takes the value $\delta_1 = \lambda+\frac 23$. The error in \eqref{eq:gK1_error1} makes sense, because the error in \Cref{prp:sanr} is uniform in $\bm u$.}
	\begin{multline}\label{eq:gK1_error1}
		g_{\a,\b;N,\bm\ell}^{[K,1]}(z) = \left(\frac z\pi\right)^\frac N2\frac{e^\frac{\pi^2N}{12z}}{2^{K+\frac12}}\sum_{r=0}^{R-1} \sum_{\bm u \in \mc T_{\alpha,\beta;N,\bm\ell}\cap \mc U_{\varepsilon,\lambda}} e^{-\bm u^T\bm u} C_r(\bm u)z^\frac r2\\
		+ z^\frac N2e^\frac{\pi^2N}{12z}O\rb{\varepsilon^{3R\rb{\lambda+\frac 23}}} \sum_{\bm u \in \mc T_{\alpha,\beta;N,\bm\ell}\cap \mc U_{\varepsilon,\lambda}} e^{-\bm u^T\bm u}.
	\end{multline}
	Since the summation over $\Tc_{\a,\b;N,\bm\ell}\cap\Uc_{\e,\l}$ contains $\ll\e^{N\l}$ terms, it follows from \Cref{lem:eb} that the second term on the right-hand side of \eqref{eq:gK1_error1} is bounded as claimed.
\end{proof}

For convenience, we set for $j\in\N_0$
\begin{equation*}
	G_j(\bm u) := e^{\frac{\pi^2N}{12z}} C_j(\bm u) e^{-\bm u^T\bm u},
\end{equation*} 
where the polynomial $C_j(\bm u)$ is defined in \eqref{eq:Crud}; in particular, we have $\deg(C_j)\le3j$ (see the proof of \Cref{prp:sanr}). We next estimate $\sum_{\bm u\in\Tc_{\a,\b;N,\bm\ell}\cap\Uc_{\e,\l}}G_j(\bm u)$.

\begin{lem}\label{lem:Gjua}
	If $y\ll \varepsilon^{1+\lambda+\delta}$ for some $\delta > 0$, then, for $R\in\N$, we have, as $z\to 0$,
	\begin{equation*}
		\sum_{\bm u\in\Tc_{\a,\b;N,\bm\ell}\cap\Uc_{\e,\l}} G_j(\bm u) = e^\frac{\pi^2N}{12z}\sum_{r=-N}^{R-1} V_{j,r}z^\frac r2 + O\left(\e^{\left(\l+\frac12\right)(3j+2N+R-1)+\frac{N+R}{2}-1}e^\frac{\pi^2N}{12\e}\right).
	\end{equation*}
\end{lem}

\begin{proof}
	For the first step, we consider the sum over $u_\a$, the $\a$-th component of $\bm u$. To emphasize this component, we write $\bm{u_{[1]}}$ to denote the remaining $N-1$ components of $\bm u$, and $\bm u = (\bm{u_{[1]}}, u_\a)$. Note that this notation does not say that $u_\a$ is the $N$-th component of $\bm u$.

	We decompose
	\begin{equation}\label{eq:sGjud} 
		\sum_{\bm u\in\Tc_{\a,\b;N,\bm\ell}\cap\Uc_{\e,\l}} G_j(\bm u) = \sum_{\substack{\bm u_{\bm{[1]}}\in\C^{N-1}\\\exists u_\a\in\C,\left(\bm u_{\bm{[1]}},u_\a\right)\in\Tc_{\a,\b;N,\bm\ell}\cap\Uc_{\e,\l}}} \sum_{\substack{u_\a\in\mf u_\a\left(\bm{u_{[1]}}\right)\\\left(\bm u_{\bm{[1]}},u_\a\right)\in\Uc_{\e,\l}}} G_j\left(\bm{u_{[1]}},u_\a\right),
	\end{equation}
	where
	\begin{equation*}
		\mf u_\alpha\left(\bm{u_{[1]}}\right) := \cbd{u_\beta+t\sqrt{z}}{t\in {[\ell_\alpha-\ell_\beta]_N}+N\N_0}. 
	\end{equation*}
	If $\bm{u_{[1]}}$ is such that there exists $u_\alpha \in \C$ with $(\bm{u_{[1]}},u_\a) \in \mc T_{\a,\b;N,\bm\ell} \cap \mc U_{\varepsilon,\lambda}$, then we have
	\ba
		\mf u_\alpha\rb{\bm{u_{[1]}}} = \cb{u_\alpha\in\C : \rb{\bm u_{\bm{[1]}},u_\alpha}\in\mc T_{\alpha,\beta;N,\bm\ell}},
	\ea
	so the decomposition \eqref{eq:sGjud} makes sense. We evaluate the sum variable by variable, starting with the innermost sum over $u_\a$. By \Cref{lem:uoor}, we may extend the inner sum in \eqref{eq:sGjud} by removing the condition $(\bm{u_{[1]}},u_\a)\in\Uc_{\e,\l}$. This introduces an error term of size $O(\e^L)e^\frac{\pi^2N}{12z}$ for all $L\in\N$, which is negligible. By \Cref{prp:EMF} with $a=u_\b+[\ell_\a-\ell_\b]_N\sqrt z$, $z\mapsto N\sqrt z$, we obtain
	\begin{align}\nonumber 
		\sum_{u_\a\in\mf u_\a\left(\bm{u_{[1]}}\right)} G_j\left(\bm{u_{[1]}},u_\a\right) &= \frac{1}{N\sqrt z}\int_{u_\b+[\ell_\a-\ell_\b]_N\sqrt z}^{u_\b+\sqrt z\infty} G_j\left(\bm{u_{[1]}},u_\a\right) du_\a + \frac12G_j\left(\bm{u_{[1]}},u_\b+[\ell_\a-\ell_\b]_N\sqrt z\right)\\
		\nonumber
		&\hspace{1.5cm}- \sum_{r=1}^R \frac{B_{2r}N^{2r-1}z^{r-\frac12}}{(2r)!}G_j^{(2r-1)}\left(\bm{u_{[1]}},u_\b+[\ell_\a-\ell_\b]_N\sqrt z\right)\\
		\label{eq:s1GjEMFf}
		&\hspace{3cm}+ O(1)z^R\int_{u_\b+[\ell_\a-\ell_\b]_N\sqrt z}^{u_\b+\sqrt z\infty} \vb{G_j^{(2R+1)}\left(\bm{u_{[1]}},u_\a\right)} du_\a.
	\end{align}
	
	First we estimate the second integral in \eqref{eq:s1GjEMFf}. By a direct calculation, we see that, for $k\in\N$
	\ba
		G_j^{(k)} \rb{\bm{u_{[1]}}, u_\alpha} = e^{\frac{\pi^2N}{12z}} P_k(\bm u) e^{-\bm u^T\bm u},
	\ea
	where the derivative is taken with respect to $u_\a$ and where $P_k(\bm u)$ is a polynomial of degree at most $3j+k$. By \Cref{lem:uoor}, we may restrict the second integral in \eqref{eq:s1GjEMFf} to $\Uc_{\e,\l}$, introducing an error of size $O(\e^L)e^\frac{\pi^2N}{12z}$ for all $L\in\N$. Now we consider the part of the integral over $\Uc_{\e,\l}$, which has measure $\ll\e^{\l+\frac12}$. For $\bm u\in\Uc_{\e,\l}$, we use $|\bm u|\ll\e^{\l+\frac12}$ and \Cref{lem:eb} to conclude that
	\begin{equation}\label{eq:s1Gjkd} 
		G_j^{(k)} \rb{\bm{u_{[1]}}, u_\alpha} \ll \varepsilon^{\rb{\lambda+\frac 12}\rb{3j+k}}e^{\frac{\pi^2N}{12\varepsilon}}.
	\end{equation}
	Hence
	\begin{equation}\label{eq:s1Gjeb} 
		\int_{u_\b+[\ell_\a-\ell_\b]_N\sqrt z}^{u_\b+\sqrt z\infty} \vb{G_j^{(2R+1)}\left(\bm{u_{[1]}},u_\a\right)} du_\a \ll \e^{\left(\l+\frac12\right)(3j+2R+2)}e^\frac{\pi^2N}{12\e}.
	\end{equation}
	
	Next we consider the first integral in \eqref{eq:s1GjEMFf}. We split 
	\begin{equation*}
		\int\limits_{u_\b+[\ell_\a-\ell_\b]_N\sqrt z}^{u_\b+\sqrt z\infty}\hspace{-.1cm} G_j\left(\bm{u_{[1]}},u_\a\right) du_\a = \hspace{-.1cm}\int\limits_{u_\b}^{u_\b+\sqrt z\infty}\hspace{-.1cm} G_j\left(\bm{u_{[1]}},u_\a\right) du_\a - \hspace{-.1cm}\int\limits_{u_\b}^{u_\b+[\ell_\a-\ell_\b]_N\sqrt z}\hspace{-.1cm} G_j\left(\bm{u_{[1]}},u_\a\right) du_\a.
	\end{equation*}
	We keep the first integral on the right-hand side. For the second integral, we apply \Cref{prp:EMF} (with $a=u_\b$, $z\mapsto[\ell_\a-\ell_\b]_N\sqrt z$) to obtain
	\begin{multline}\label{eq:s1GjrEMF} 
		\int_{u_\beta}^{u_\beta+[\ell_\alpha-\ell_\beta]_N\sqrt{z}} G_j\rb{\bm u_{\bm{[1]}},u_\alpha} du_\alpha = \frac{[\ell_\alpha-\ell_\beta]_N\sqrt{z}}{2} \rb{G_j\rb{\bm u_{\bm{[1]}},u_\beta+[\ell_\alpha-\ell_\beta]_N\sqrt{z}} + G_j\rb{\bm u_{\bm{[1]}},u_\beta}}\\
		- \sum_{r=1}^R \frac{B_{2r} [\ell_\alpha-\ell_\beta]_N^{2r} z^r}{(2r)!} \rb{G_j^{(2r-1)}\rb{\bm u_{\bm{[1]}},u_\beta+[\ell_\alpha-\ell_\beta]_N\sqrt{z}} - G_j^{(2r-1)}\rb{\bm u_{\bm{[1]}},u_\beta}}\\
		+ O(1)z^{R+\frac 12} \int_{u_\beta}^{u_\beta+[\ell_\alpha-\ell_\beta]_N\sqrt{z}} \vb{G_j^{(2R+1)}\rb{\bm u_{\bm{[1]}},u_\alpha}} du_\alpha.
	\end{multline}
	Using \eqref{eq:s1Gjkd}, we conclude
	\begin{equation}\label{eq:s1Gjreb} 
		\int_{u_\beta}^{u_\beta+[\ell_\alpha-\ell_\beta]_N\sqrt{z}} \vb{G_j^{(2R+1)}\rb{\bm u_{\bm{[1]}},u_\alpha}} \ll \varepsilon^{\rb{\lambda+\frac 12}\rb{3j+2R+1} + \frac 12}e^{\frac{\pi^2N}{12\varepsilon}}. 
	\end{equation}
	It remains to consider terms of the form $G_j^{(k)}(\bm u_{\bm{[1]}},u_\b+[\ell_\a-\ell_\b]_N\sqrt z)$ appearing in \eqref{eq:s1GjEMFf} and \eqref{eq:s1GjrEMF}. We apply Taylor's Theorem and rewrite
	\begin{multline}\label{eq:s1GjTs} 
		G_j^{(k)}\rb{\bm u_{\bm{[1]}},u_\beta+[\ell_\alpha-\ell_\beta]_N\sqrt{z}} = \sum_{r=0}^{R-1} \frac{[\ell_\alpha-\ell_\beta]_N^r z^{\frac r2}}{r!} G_j^{(k+r)}\rb{\bm u_{\bm{[1]}},u_\beta}\\
		+ \frac{1}{(R-1)!} \int_{u_\beta}^{u_\beta+[\ell_\alpha-\ell_\beta]_N\sqrt{z}} G_j^{(k+R)}\rb{\bm u_{\bm{[1]}},u_\alpha} (u_\alpha-u_\beta)^{R-1} du_\alpha.
	\end{multline}
	Using \eqref{eq:s1Gjkd}, we conclude that
	\begin{equation}\label{eq:s1GjTe} 
		\int_{u_\beta}^{u_\beta+[\ell_\alpha-\ell_\beta]_N\sqrt{z}} G_j^{(k+R)}\rb{\bm u_{\bm{[1]}},u_\alpha} (u_\alpha-u_\beta)^{R-1} du_\alpha \ll \varepsilon^{\rb{\lambda+\frac 12}\rb{3j+k+R} + \frac{R}2}e^{\frac{\pi^2N}{12\varepsilon}}.
	\end{equation}
	Combining \eqref{eq:s1GjEMFf}, \eqref{eq:s1GjrEMF}, \eqref{eq:s1GjTs}, and applying the error estimates \eqref{eq:s1Gjeb}, \eqref{eq:s1Gjreb}, and \eqref{eq:s1GjTe}, we write
	\begin{multline}\label{eq:s1Gja} 
		\sum_{u_\a\in\mf u_\a\left(\bm{u_{[1]}}\right)} G_j\left(\bm{u_{[1]}},u_\a\right) = \frac{1}{N\sqrt z}\int_{u_\b}^{u_\b+\sqrt z\infty} G_j\left(\bm{u_{[1]}},u_\a\right) du_\a + \sum_{r=0}^{R-1} W_r G_j^{(r)}\left(\bm{u_{[1]}},u_\b\right)z^\frac r2\\
		+ O\left(\e^{\left(\l+\frac12\right)(3j+R+1)+\frac{R-1}{2}}e^\frac{\pi^2N}{12\e}\right),
	\end{multline}
	where
	\begin{align}\label{eq:W0} 
	W_0 &= \frac 12 - \frac{[\ell_\a-\ell_\b]_N}N,\\
	W_r &= \rb{\frac 12-\frac{[\ell_\a-\ell_\b]_N}{2N}}\frac{[\ell_\a-\ell_\b]_N^r}{r!} + \sum_{t=1}^{\lceil\frac r2\rceil} \frac{B_{2t}\left([\ell_\a-\ell_\b]_N^{2t} - N^{2t}\right)[\ell_\a-\ell_\b]_N^{r-2t+1}}{(2t)!(r-2t+1)!N}\nonumber\\
	&\hspace{1cm} - \d_{r\equiv 1\pmod{2}} \frac{B_{r+1}[\ell_\a-\ell_\b]_N^{r+1}}{(r+1)!N}, & &(r\ge 1).\nonumber
	\end{align}

	Next we consider the summation over the other variables. We set
	\begin{equation*}
		G_{j,\a}\left(\bm{u_{[1]}}\right) := \sum_{u_\a\in\mf u_\a\left(\bm{u_{[1]}}\right)} G_j\left(\bm{u_{[1]}},u_\a\right).
	\end{equation*}
	Let $1\le c\le N$, $c\ne\a$, and consider the summation over $u_c$. To emphasize the $c$-th and the $\a$-th entries of $\bm u$, we write $\bm{u_{[2]}}$ to denote the remaining $N-2$ components of $\bm u$, and write $\bm u=(\bm{u_{[2]}},u_c,u_\a)$. We are again abusing notation and this does not mean that $u_c,u_\a$ are the final components of $\bm u$. We write $\bm{u_{[1]}}=(\bm{u_{[2]}},u_c)$. We consider
	\[
		\sum_{\bm u\in\Tc_{\a,\b;N,\bm\ell}\cap\Uc_{\e,\l}} G_j(\bm u) = \sum_{\substack{\bm u_{\bm{[2]}}\in\C^{N-2}\\\exists u_c,u_\a\in\C,\left(\bm u_{\bm{[2]}},u_c,u_\a\right)\in\Tc_{\a,\b;N,\bm\ell}\cap\Uc_{\e,\l}}} \sum_{\substack{u_c\in\mf u_c,\ u_\a\in\mf u_\a\left(\bm{u_{[2]}},u_c\right)\\\left(\bm u_{\bm{[2]}},u_c,u_\a\right)\in\Uc_{\e,\l}}} G_j\left(\bm{u_{[1]}},u_\a\right),
	\]
	where 
	\begin{equation*}
		\mf u_c := \cbd{-\frac{\log (2)}{\sqrt{z}} + t\sqrt{z}}{t\in\ell_c+N\N_0}.
	\end{equation*}
	Analogous to the summation over $u_\a$, we may use \Cref{lem:uoor} to extend the sum by removing the condition $(\bm{u_{[2]}},u_c,u_\a)\in\Uc_{\e,\l}$, introducing an error of size $O(\e^L)e^\frac{\pi^2N}{12z}$ for all $L\in\N$, which is negligible. So we may consider instead the sum
	\[
		\sum_{\substack{u_c\in\mf u_c\\u_\a\in\mf u_\a\left(\bm{u_{[2]}},u_c\right)}} G_j\left(\bm{u_{[2]}},u_c,u_\a\right) = \sum_{u_c\in\mf u_c} G_{j,\a}\left(\bm{u_{[2]}},u_c\right).
	\]
	Again, we apply \Cref{prp:EMF}, and write
	\begin{multline}\label{eq:s2GjEMF} 
		\sum_{u_c\in\mf u_c} G_{j,\a}\left(\bm{u_{[2]}},u_c\right) = \frac{1}{N\sqrt z}\int_{-\frac{\log(2)}{\sqrt z}+\ell_c\sqrt z}^{-\frac{\log(2)}{\sqrt z}+\sqrt z\infty} G_{j,\a}\left(\bm{u_{[2]}},u_c\right) du_c + \frac12G_{j,\a}\left(\bm{u_{[2]}},-\frac{\log(2)}{\sqrt z}+\ell_c\sqrt z\right)\\
		- \sum_{r=1}^R \frac{B_{2r}N^{2r-1}z^{r-\frac12}}{(2r)!}G_{j,\a}^{(2r-1)}\left(\bm{u_{[2]}},-\frac{\log(2)}{\sqrt z}+\ell_c\sqrt z\right)\\
		+ O(1)z^R\int_{-\frac{\log(2)}{\sqrt z}+\ell_c\sqrt z}^{-\frac{\log(2)}{\sqrt z}+\sqrt z\infty} \vb{G_{j,\a}^{(2R+1)}\left(\bm{u_{[2]}},u_c\right)} du_c.
	\end{multline}
Thanks to the exponential decay of $e^{-\bm u^T\bm u}$, we have, as $z\to0$ for all $r\in\N_0$ and $L\in\N$,
	\begin{equation}\label{eq:s2Gjs}
		\scalebox{.97}{$\displaystyle G_{j,\a}^{(r)}\left(\bm{u_{[2]}},-\frac{\log(2)}{\sqrt z}+\ell_c\sqrt z\right) = O\left(\e^L\right)e^\frac{\pi^2N}{12z},\ \int_{-\frac{\log(2)}{\sqrt z}-\sqrt z\infty}^{-\frac{\log(2)}{\sqrt z}+\ell_c\sqrt z} G_{j,\a} \left(\bm{u_{[2]}},u_c\right) du_c = O\left(\e^L\right)e^\frac{\pi^2N}{12z}.$}
	\end{equation}
	Therefore, the lower boundary of the first integral in \eqref{eq:s2GjEMF} can be extended to $-\frac{\log(2)}{\sqrt z}-\sqrt z\infty$, and all the terms except for the main term can be ignored introducing an error term of size $O(\e^L)e^\frac{\pi^2N}{12z}$ for all $L\in\N$. Meanwhile, the error term in \eqref{eq:s2GjEMF} has size
	\begin{equation}\label{eq:s2Gjeb} 
		z^R\int_{-\frac{\log(2)}{\sqrt z}+\ell_c\sqrt z}^{-\frac{\log(2)}{\sqrt z}+\sqrt z\infty} \vb{G_{j,\a}^{(2R+1)}\left(\bm{u_{[2]}},u_c\right)} du_c = O\left(\e^{\left(\l+\frac12\right)(3j+2R+2)+R+\l}e^\frac{\pi^2N}{12\e}\right).
	\end{equation}
	By taking $R$ sufficiently large, it follows from \eqref{eq:s2GjEMF}, \eqref{eq:s2Gjs}, and \eqref{eq:s2Gjeb} that
	\ba
		\sum_{u_c\in\mf u_c} G_{j,\a}\rb{\bm{u_{[2]}}, u_c} = \frac{1}{N\sqrt{z}} \int_{-\frac{\log(2)}{\sqrt{z}}+\R\sqrt{z}} G_{j,\a}\rb{\bm{u_{[2]}}, u_c} du_c + O\rb{\varepsilon^Le^{\frac{\pi^2N}{12\varepsilon}}}
	\ea
	for all $L\in\N$. Using the same argument, we sum over the other coordinates, and obtain that
	\begin{equation*}
		\sum_{\bm u\in\Tc_{\a,\b;N,\bm\ell}\cap \mc U_{\varepsilon,\lambda}} G_j(\bm u) = \frac{z^\frac{1-N}{2}}{N^{N-1}}\int_{\rb{-\frac{\log(2)}{\sqrt{z}}+\R\sqrt{z}}^{N-1}} G_{j,\a}\left(\bm{u_{[1]}}\right) \bm{du_{[1]}} + O\rb{\varepsilon^Le^{\frac{\pi^2N}{12\varepsilon}}}
	\end{equation*}
	for all $L\in\N$. Again, we may use \Cref{lem:uoor} to restrict the integral to
	\begin{equation*}
		\Uc_{\bm{[1]}} := \left\{\bm{u_{[1]}}\in\left(-\frac{\log(2)}{\sqrt z}+\R\sqrt z\right)^{N-1} : \left|\bm{\mu_{[1]}}\right|\le\e^\l\right\},
	\end{equation*}
	where $\bm{\mu_{[1]}}$ is the $(N-1)$-tuple associated to $\bm{u_{[1]}}$ via \eqref{eq:uje}, with a negligible error. So we may write
	\begin{equation*}
		\sum_{\bm u\in\Tc_{\a,\b;N,\bm\ell}\cap \mc U_{\varepsilon,\lambda}} G_j(\bm u) = \frac{z^\frac{1-N}{2}}{N^{N-1}}\int_{\mc U_{\bm{[1]}}} G_{j,\a}\left(\bm{u_{[1]}}\right) \bm{du_{[1]}} + O\rb{\varepsilon^Le^{\frac{\pi^2N}{12\varepsilon}}}.
	\end{equation*}
	Applying the asymptotic expansion \eqref{eq:s1Gja}, we get
	\begin{multline}\label{eq:Gjsaas} 
		\sum_{\bm u\in\Tc_{\a,\b;N,\bm\ell}\cap\Uc_{\e,\l}} G_j(\bm u) = \frac{z^{-\frac N2}}{N^N}\int_{\mc U_{\bm{[1]}}} \int_{u_\b}^{u_\b+\sqrt z\infty} G_j\left(\bm{u_{[1]}},u_\a\right) du_\a \bm{du_{[1]}}\\
		+ \sum_{r=0}^{R-1} \frac{z^\frac{1-N+r}{2}}{N^{N-1}}W_r\int_{\mc U_{\bm{[1]}}} G_j^{(r)}\left(\bm{u_{[1]}},u_\b\right) \bm{du_{[1]}} + O\left(\e^{\left(\l+\frac12\right)(3j+R+1)+\frac{R-1}{2}}e^\frac{\pi^2N}{12\e}\right)\int_{\mc U_{\bm{[1]}}} \bm{du_{[1]}}.
	\end{multline}
	Since $\mc U_{\bm{[1]}}$ has measure $\ll\e^{(N-1)(\l+\frac12)}$, the error term in \eqref{eq:Gjsaas} has size
	\[
		O\left(\e^{\left(\l+\frac12\right)(3j+R+N)+\frac{R-1}{2}}e^\frac{\pi^2N}{12\e}\right).
	\]
	Now, we may use \Cref{lem:uoor} again, to extend the integrals in \eqref{eq:Gjsaas} to $(-\frac{\log(2)}{\sqrt{z}}+\R\sqrt{z})^{N-1}$
	\begin{multline*}
		\sum_{\bm u\in\Tc_{\a,\b;N,\bm\ell}\cap\Uc_{\e,\l}} G_j(\bm u) = \frac{z^{-\frac N2}}{N^N}\int_{\left(-\frac{\log(2)}{\sqrt z}+\R\sqrt z\right)^{N-1}} \int_{u_\b}^{u_\b+\sqrt z\infty} G_j\left(\bm{u_{[1]}},u_\a\right) du_\a \bm{du_{[1]}}\\
		+ \sum_{r=0}^{R-1} \frac{z^\frac{1-N+r}{2}}{N^{N-1}}W_r \int_{\left(-\frac{\log(2)}{\sqrt z}+\R\sqrt z\right)^{N-1}} G_j^{(r)}\left(\bm{u_{[1]}},u_\b\right) \bm{du_{[1]}} + O\left(\e^{\left(\l+\frac12\right)(3j+R+N)+\frac{R-1}{2}}e^\frac{\pi^2N}{12\e}\right).
	\end{multline*}
	Finally, since $G_j^{(r)}(\bm{u_{[1]}},u_\a)$ is holomorphic and has rapid decay, we may shift the path and write
	\begin{multline*}
		\sum_{\bm u\in\Tc_{\a,\b;N,\bm\ell}\cap \mc U_{\varepsilon,\lambda}} G_j(\bm u) = \frac{z^{-\frac N2}}{N^N} \int_{\mc R_{\alpha,\beta;N}} G_j\rb{\bm u} \bm{du}\\
		+ \sum_{r=0}^{R-1} \frac{z^\frac{1-N+r}{2}}{N^{N-1}} W_r \int_{\R^{N-1}} G_j^{(r)} \rb{\bm{u_{[1]}}, u_\beta} \bm{du_{[1]}} + O\rb{\varepsilon^{\rb{\lambda+\frac 12}(3j+R+N)+\frac{R-1}2}e^{\frac{\pi^2N}{12\varepsilon}}},
	\end{multline*}
	where $\Rc_{\a,\b;N}:=\{\bm u\in\R^N:u_\a\ge u_\b\}$. So we may set
	\begin{align}\label{eq:VjN}
		V_{j,-N} &:= \frac{e^{-\frac{\pi^2N}{12z}}}{N^N}\int_{\Rc_{\a,\b;N}} G_j(\bm u) \bm{du},\\
		\label{eq:Vj1N}
		V_{j,r} &:= \frac{e^{-\frac{\pi^2N}{12z}}}{N^{N-1}}W_{r+N-1} \int_{\R^{N-1}} G_j^{(r+N-1)}\left(\bm{u_{[1]}},u_\b\right) \bm{du_{[1]}} \qquad (r\ge-N+1).
	\end{align}
Note that $V_{j,r}$ does not depend on $z$.
\end{proof}

We are now ready to prove \Cref{prp:gK1a}.

\begin{proof}[Proof of \Cref{prp:gK1a}]
	We estimate the following expression, occuring in \Cref{lem:gK1sa}:
	\[
		\leg z\pi^\frac N2\frac{e^\frac{\pi^2N}{12z}}{2^{K+\frac12}}\sum_{r=0}^{R-1} \sum_{\bm u\in\Tc_{\a,\b;N,\bm\ell}\cap\Uc_{\e,\l}} e^{-\bm u^T\bm u}C_r(\bm u)z^\frac r2 = \frac{1}{2^{K+\frac12}\pi^\frac N2}\sum_{r=0}^{R-1} z^\frac{N+r}{2}\sum_{\bm u\in\Tc_{\a,\b;N,\bm\ell}\cap\Uc_{\e,\l}} G_r(\bm u).
	\]
	By \Cref{lem:Gjua}, we have
	\ba
		z^{\frac{N+r}{2}} \sum_{\bm u \in \mc T_{\alpha,\beta;N,\bm\ell}\cap \mc U_{\varepsilon,\lambda}} G_r\rb{\bm u} = e^{\frac{\pi^2N}{12z}} \sum_{j=0}^{R-1-r} V_{r,j-N} z^{\frac{r+j}2} + O\rb{\varepsilon^{\rb{\lambda+\frac 12}\rb{2r+N+R-1} + \frac{N+R}{2}-1}e^{\frac{\pi^2N}{12\varepsilon}}}.
	\ea
	Summing over $r$ gives
	\begin{equation*}
		\sum_{r=0}^{R-1} z^{\frac{N+r}{2}} \sum_{\bm u \in \mc T_{\alpha,\beta;N,\bm\ell}\cap \mc U_{\varepsilon,\lambda}} G_r\rb{\bm u} = e^{\frac{\pi^2N}{12z}} \sum_{r=0}^{R-1} \sum_{j=0}^r V_{j,r-j-N} z^{\frac r2} + O\rb{\varepsilon^{3(R-1)\rb{\lambda+\frac 23} + N\rb{\lambda+1}-\frac 12}e^{\frac{\pi^2N}{12\varepsilon}}}. 
	\end{equation*}
	Since $N\ge 2$ and $-\frac 23 < \lambda < -\frac 12$, the error term $O(\varepsilon^{N(\lambda+\frac 12)+3R(\lambda+\frac 23)}e^{\frac{\pi^2N}{12\varepsilon}})$ from \Cref{lem:gK1sa} dominates. Setting 
	\begin{equation}\label{eq:Elrd} 
		E_{\bm\ell,r} := \sum_{j=0}^r V_{j,r-j-N}
	\end{equation}
	gives the proposition.
\end{proof}

Now we show that the asymptotic expansion above can also be applied for larger values of $|y|$, with negligible error. We require the following technical lemma about $\La(y)$.

\begin{lem}\label{lem:sy} 
	Let $s(y):=\re(\frac{\La(y)}{1+iy}-\frac{\pi^2N}{12})$. Then we have the following:
	\begin{enumerate}[leftmargin=*,label=\rm{(\arabic*)}]
		\item We have $s(y)\le 0$ for all $y\in\R$, and the equality holds if and only if $y=0$.
		
		\item For any $y_0>0$, there exists $d>0$ such that $s(y)<-d$ for all $|y|\ge y_0$.
		
		\item We have, as $y\to0$,
		\ba
			s(y) = N\rb{\log (2)^2 -\frac{\pi^2}{12}}{y^2} + O\rb{y^4}.
		\ea
	\end{enumerate}
\end{lem}

\begin{proof}
	(1) and (2) are easily verified, and (3) is obtained by evaluating the Taylor series at $y=0$.
\end{proof}

Next we show that the asymptotic expansion in \Cref{prp:gK1a} gives a good approximation to $g_{\a,\b;N,\bm\ell}^{[K]}(z)$ in the range $y\ll\e^{-1-\frac{3\l}2+\d}$, which covers the major arc.

\begin{prp}\label{prp:gKwr} 
	If $y\ll\e^{-1-\frac{3\l}{2}+\d}$, then, for every $L\in\N$,
	\[
		g_{\a,\b;N,\bm\ell}^{[K]}(z) = \frac{e^\frac{\pi^2N}{12z}}{2^{K+\frac12}\pi^\frac N2}\sum_{r=0}^{R_0-1} E_{\bm\ell,r} z^\frac r2 + O\rb{\varepsilon^Le^\frac{\pi^2N}{12\varepsilon}}
	\]
	as $z\to0$, where $E_{\bm\ell,r}$ and $R_0=R_0(L)$ are as in \Cref{prp:gK1a}.
\end{prp}

\begin{proof}
	We split as in \eqref{eq:gKs}. By \Cref{prp:ore}, we have, for all $L\in\N$ as $z\to0$,
	\[
		g_{\a,\b;N,\bm\ell}^{[K,2]}(z) \ll \e^L e^\frac{\pi^2N}{12\e}.
	\]
	So it remains to estimate $g_{\a,\b;N,\bm\ell}^{[K,1]}$. The proposition follows from the following claim:
	\[
		e^{-\frac{\pi^2 N}{12\e}}\left(g_{\a,\b;N,\bm\ell}^{[K,1]}(z) - \frac{e^\frac{\pi^2N}{12z}}{2^{K+\frac12}\pi^\frac N2}\sum_{r=0}^{R_0-1} E_{\bm\ell,r}z^\frac r2\right) = O\left(\e^L\right).
	\]
	From \Cref{prp:sawr}, if $y \ll\e^{-1-\frac{3\l}2+\d}$, then we have an asymptotic expansion
	\begin{align}\label{eq:gK1wra} 
		&e^{-\frac{\pi^2N}{12\e}}g_{\a,\b;N,\bm\ell}^{[K,1]}(z) = 2^{-\left(K+\frac12\right)(1+iy)}\left(1-2^{-(1+iy)}\right)^{-\frac N2}\leg{z}{2\pi}^\frac N2e^{\frac1\e\left(\frac{\La(y)}{1+iy}-\frac{\pi^2N}{12}\right)}\\
		\nonumber
		&\hspace{.3cm}\times \sum_{\bm n\in\Sc_{\a,\b;N,\bm\ell}\cap\Nc_{\e,\l}} e^{-\frac12\bm v^T\Ac\bm v+\frac{1}{\sqrt z}\left(-\log(2)(1+iy)-\Log\left(1-2^{-(1+iy)}\right)\right)\sum_{j=1}^N v_j}\sum_{r=0}^{R_0-1} D_r(\bm v,y)z^\frac r2\\
		\nonumber
		&\hspace{.6cm}+ z^\frac N2e^{\frac1\e\left(\frac{\La(y)}{1+iy}-\frac{\pi^2N}{12}\right)}O\left(\e^{2R_0\d_2}\right) \sum_{\bm n\in\Sc_{\a,\b;N,\bm\ell}\cap\Nc_{\e,\l}} e^{-\frac12\bm v^T\Ac\bm v+\frac{1}{\sqrt z}\left(-\log(2)(1+iy)-\Log\left(1-2^{-(1+iy)}\right)\right)\sum_{j=1}^N v_j}.
	\end{align}
	We claim that if $\varepsilon^{\frac 12-\delta} \ll y\ll \varepsilon^{-1-\frac {3\lambda}2+\delta}$ for some $\delta>0$, then we have, for all $L\in\N$, 
	\begin{equation}\label{eq:gK1oMa} 
		e^{-\frac{\pi^2 N}{12\varepsilon}} g_{\alpha, \beta;N,\bm\ell}^{[K,1]}(z) = O\rb{\varepsilon^L}.
	\end{equation}
	To prove \eqref{eq:gK1oMa}, it suffices to show that the exponent
	\begin{equation*}
		\frac1\e\left(\frac{\La(y)}{1+iy}-\frac{\pi^2N}{12}\right) -\frac12\bm v^T\Ac\bm v+\frac{1}{\sqrt z}\left(-\log(2)(1+iy)-\Log\left(1-2^{-(1+iy)}\right)\right)\sum_{j=1}^N v_j
	\end{equation*}
	has negative real part of size $\gg\e^{-\d_0}$ for some $\d_0>0$; the other factors in \eqref{eq:gK1wra} are bounded as $z\to0$. By \Cref{lem:sy} (1) the real part of $\frac1\e\re(\frac{\La(y)}{1+iy}-\frac{\pi^2N}{12})=\frac{s(y)}{\e}$ is negative. So it suffices to show that this exponent has sufficiently large size and dominates other exponents with positive real parts.

	First consider the case $\e^{-\d+\frac12}\ll y\ll1$ for some (sufficiently small) $\d>0$. By \Cref{lem:sy} (3), we find that
	\begin{equation*}
		\frac{s(y)}{\e} \gg \frac{y^2}{\e} \gg \e^{-2\d}
	\end{equation*}
	as $\e\to0$. Meanwhile, by computing the Taylor series expansion we have, as $y\to0$,
	\[
		\re\left( -\log(2)(1+iy)-\Log\left(1-2^{-(1+iy)}\right)\right) \ll y^2.
	\]
	As $|\bm\nu|\ll\e^\l$ (see the proof of \Cref{prp:sewr}), we conclude that
	\[
		\re\left(\frac{1}{\sqrt z}\left(-\log(2)(1+iy)-\Log\left(1-2^{-(1+iy)}\right)\right)\sum_{j=1}^N v_j\right) \ll \e^\l y^2.
	\]
	Next we consider the exponent $-\frac 12\bm v^T\Ac\bm v$. For this we split into two subcases.
	\begin{enumerate}[leftmargin=*]
		\item Suppose that $\e^{-\d+\frac12}\ll y\ll\e^\frac14$. Since $v_j$ is a real multiple of $\sqrt z$, the condition $y\ll\e^\frac14$ implies that $-\frac12\bm v^T\Ac\bm v$ has negative real part as $z\to0$.
		
		\item Suppose $\e^\frac14\ll y\ll1$. In this case we have $|\bm v|\ll\e^{\frac12+\l}$, and it follows that $-\frac12\bm v^T\Ac\bm v\ll\e^{1+2\l}$.
	\end{enumerate}
	In either case, the exponent $\frac{s(y)}{\e}$ has size $\gg\e^{-2\d}$ and dominate other exponents appearing in \eqref{eq:gK1wra} that have positive real parts. So we conclude that \eqref{eq:gK1oMa} holds if $\e^{-\d+\frac12}\ll y\ll1$, by taking $\d_0=2\d$.
	
	Next we consider the case $1\ll y\ll\e^{-1-\frac{3\l}{2}+\d}$ for some $\d>0$. We use the trivial bound
	\ba
		\re\rb{-\log (2)(1+iy) - \Log\rb{1-2^{-(1+iy)}}} \ll 1.
	\ea
	It then follows from the bound $\vb{\bm\nu}\ll \varepsilon^\lambda$ that
	\[
		\re\left(\frac{1}{\sqrt z}\left(-\log(2)(1+iy)-\Log\left(1-2^{-(1+iy)}\right)\right)\sum_{j=1}^N v_j\right) \ll \e^\l.
	\]
	Meanwhile, by \Cref{lem:sy} (2), we have $\frac{s(y)}{\e}\gg\frac1\e$. Finally, as $|\bm\nu|\ll\e^\l$ and $z\ll\e^{-\frac{3\l}{2}+\d_2}$, where $\d_2=\min\{1+\frac{3\l}{2},\d\}>0$ (see the proof of \Cref{prp:sewr}), we deduce that $|\bm v|\ll\e^{\frac\l4+\frac{\d_2}{2}}$, and hence
	\[
		-\tfrac12\bm v^T\Ac\bm v \ll \e^{\frac\l2+\d_2}.
	\]
	From the computation above, we have that $\frac{s(y)}{\e}\gg\frac1\e$, and this exponent dominates the other exponents. So \eqref{eq:gK1oMa} also holds.

	For the expression $2^{-K-\frac12}\pi^{-\frac N2}e^{\frac{\pi^2N}{12z}-\frac{\pi^2N}{12\e}}\sum_{r=0}^{R_0-1}E_{\bm\ell,r}z^\frac r2$, we only have a single exponent, namely $\frac{\pi^2N}{12z}-\frac{\pi^2N}{12\varepsilon}$. This exponent has negative real part of size 
	\[
		\re\left(\frac{\pi^2N}{12z}-\frac{\pi^2N}{12\e}\right) = \frac{\pi^2N}{12\e}\rb{\re\left(\frac{1}{1+iy}\right)-1} \gg \e^{\max\{1+2\l-2\d,-1\}}
	\]
	for $\e^{\frac 12-\d}\ll y\ll\e^{-1-\frac{3\l}{2}+\d}$. It follows that, for $\e^{\frac 12-\d}\ll y\ll\e^{-1-\frac{3\l}{2}+\d}$, we have, for all $L\in\N$,
	\[
		\frac{e^{\frac{\pi^2N}{12z}-\frac{\pi^2N}{12\varepsilon}}}{2^{K+\frac12}\pi^\frac N2} \sum_{r=0}^{R_0-1} E_{\bm\ell,r}z^\frac r2 = O\left(\e^L\right).
	\]

	On the other hand, if $y\ll\e^{1+\l+\d}$, then, by \Cref{prp:gK1a}, we have, for all $L\in\N$,
	\[
		g_{\a,\b;N,\bm\ell}^{[K,1]}(z) - \frac{e^\frac{\pi^2N}{12z}}{2^{K+\frac12}\pi^\frac N2}\sum_{r=0}^{R_0-1} E_{\bm\ell,r}z^\frac r2 \ll \e^Le^\frac{\pi^2N}{12\e}.
	\]
	Choosing $\d>0$ sufficiently small, the two cases cover the full range $y\ll\e^{-1-\frac{3\l}{2}+\d}$, establishing the proposition.
\end{proof}

We are now ready to derive the asymptotic expansion of $d_{\a,\b;N}^{[K]}(n)$.

\begin{thm}\label{thm:da} 
	We have for $R\in\N$, as $n\to\infty$,
	\begin{equation*}
		d_{\a,\b;N}^{[K]}(n) = \frac{e^{\pi\sqrt\frac n3}}{2^{K+\frac12}\pi^\frac N2}\left(\sum_{r=0}^{R-1} \sum_{\substack{\bm\ell\in(\Z/N\Z)^N\\NH(\bm\ell)\equiv n\Pmod N}} \sum_{j=0}^{\Flo{\frac r2}} c_{\frac r2-j,\frac\pi2\sqrt\frac N3,j}E_{\bm\ell,r-2j}\leg Nn^\frac{r+3}{4}+O\left(n^{-\frac{R+3}{4}}\right)\right),
	\end{equation*}
	where $E_{\bm\ell,r}$ are as in \Cref{prp:gK1a} and
	\[
		c_{{A,B},r} := \frac{\left(-\frac{1}{4B}\right)^rB^{A+\frac12}{\Ga\left(A+r+\frac32\right)}}{2\sqrt\pi{r!\Ga\left(A-r+\frac32\right)}} .
	\]
\end{thm}

\begin{proof}
	By Cauchy's Theorem we have, for $n\in\N_0$,
	\ba
		d_{\alpha,\beta;N}^{[K]}(n) = \frac{1}{2\pi i}\int_\mc C \frac{\mc D_{\alpha,\beta;N}^{[K]}(q)}{q^{n+1}} dq,
	\ea
	where $\Cc$ is a circle centred at the origin inside the unit circle surrounding zero exactly once counter-clockwise. Using \eqref{eq:gfDd} and the change of variables $q=e^{-z}$, we write, for any $\e>0$, with $\z_N:=e^\frac{2\pi i}{N}$
	\begin{equation}\label{eq:dc} 
		d_{\a,\b;N}^{[K]}(n) = \frac1N\sum_{k=0}^{N-1} \sum_{\bm\ell\in(\Z/N\Z)^N} \z_N^{(n-NH(\bm\ell))k}\frac{1}{2\pi i}\int_{\e-\pi i}^{\e+\pi i} g_{\a,\b;N,\bm\ell}^{[K]}(z)e^\frac{nz}{N} dz.
	\end{equation}
	
	Let $\th>0$ be fixed. We split the integral into the {\it major arc} $\Cc_1(\e):=\{z=\e(1+iy):|y|\le\th\}$ and the {\it minor arc} $\Cc_2(\e):=\{z=\e(1+iy):\th<|y|\le\frac\pi\e\}$. Note that \Cref{prp:gKwr} applies for the whole major arc (because $\e^{-1-\frac{3\l}{2}+\d}\gg1$). So on the major arc, for every $L\in\N$, we have
	\begin{equation}\label{eq:Mae} 
		g_{\a,\b;N,\bm\ell}^{[K]}(z) - \frac{e^\frac{\pi^2N}{12z}}{2^{K+\frac12}\pi^\frac N2}\sum_{r=0}^{R_0-1} E_{\bm\ell,r}z^\frac r2 \ll \e^Le^{\frac{\pi^2N}{12\e}},
	\end{equation}
	where $R_0=R_0(L)$ is given as in \Cref{prp:gK1a}. On the minor arc, we use \Cref{prp:gKwr,prp:mae} to obtain for all $L\in\N$,
	\begin{equation}\label{eq:mae} 
		g_{\a,\b;N,\bm\ell}^{[K]}(z) \ll \e^Le^\frac{\pi^2N}{12\varepsilon}.
	\end{equation}
	
	Let $A\ge0$, $B>0$. By \cite[Lemma 3.7]{NR2017}, we have, as $n\to\infty$,
	\ba
		\frac{1}{2\pi i}\int_{\mc C_1\rb{\frac{B}{\sqrt{n}}}} z^Ae^{\frac{B^2}{z}+nz} dz = n^{\frac14(-2A-3)} e^{2B\sqrt{n}} \Bigg(\sum_{r=0}^{R-1} \frac{c_{A,B,r}}{n^{\frac r2}} + O\rb{n^{-\frac R2}}\Bigg).
	\ea
	Hence we obtain the asymptotic expansion
	\begin{multline}\label{eq:Elria}
		\frac{2^{-K-\frac12}\pi^{-\frac N2}}{2\pi i}\int_{\Cc_1\left(\frac{\pi N}{2\sqrt{3n}}\right)} e^\frac{\pi^2N}{12z}\sum_{r=0}^{R-1} E_{\bm\ell,r}z^\frac r2e^\frac{nz}{N} dz\\
		= \frac{e^{\pi\sqrt\frac n3}}{2^{K+\frac12}\pi^\frac N2}\left(\sum_{r=0}^{R-1} \sum_{j=0}^{\Flo{\frac r2}} c_{\frac r2-j,\frac\pi2\sqrt\frac N3,j} E_{\bm\ell,r-2j}\left(\frac Nn\right)^\frac{r+3}{4}+O\left(n^{-\frac{R+3}{4}}\right)\right). 
	\end{multline}
	Meanwhile, if a function $h$ satisfies $h(z) \ll \e^Le^{\frac{\pi^2 N}{12\varepsilon}}$ as $\varepsilon \to 0$ for $L\ge 0$, then we have, as $n\to\infty$,
	\begin{equation}\label{eq:hi} 
		\frac{1}{2\pi i}\int_{\frac{\pi N}{2\sqrt3n}-\pi i}^{\frac{\pi N}{2\sqrt3n}+\pi i} h(z)e^\frac{nz}{N} dz = O\left(n^{-\frac L2}e^{\pi\sqrt\frac n3}\right),
	\end{equation}
	and we use \eqref{eq:hi} to evaluate the minor arc integral and the error integral.
	
	Let $R\in\N$. We set $L=\frac{R+3}{2}$, and we pick $R_0=R_0(L)$ as in \Cref{prp:gK1a}, and write
	\begin{multline*}
		\frac{1}{2\pi i}\int_{\mc C\rb{\frac{\pi N}{2\sqrt{3n}}}} g_{\a,\b;N,\bm\ell}^{[K]}(z)e^\frac{nz}{N} dz\\
		= \frac{2^{-K-\frac12} \pi^{-\frac N2}}{2\pi i}\int_{\mc C\rb{\frac{\pi N}{2\sqrt{3n}}}} e^\frac{\pi^2N}{12z} \sum_{r=0}^{R_0-1} E_{\bm\ell,r}z^{\frac r2}e^{\frac{nz}{N}} dz + \frac{1}{2\pi i} \int_{\mc C\rb{\frac{\pi N}{2\sqrt{3n}}}} h(z) e^{\frac{nz}{N}} dz,
	\end{multline*}
	where we have $h(z)\ll\e^Le^\frac{\pi^2N}{12\e}=\e^\frac{R+3}{2}e^\frac{\pi^2N}{12\e}$ by \eqref{eq:Mae} and \eqref{eq:mae}. Using \eqref{eq:Elria} and \eqref{eq:hi}, the integrals above can be evaluated as
	\begin{align*}
		\frac{e^{\pi \sqrt{\frac n3}}}{2^{K+\frac12}\pi^\frac N2}\left(\sum_{r=0}^{R_0-1} \sum_{j=0}^{\Flo{\frac r2}} c_{\frac r2-j,\frac\pi2\sqrt{\frac{N}{3}},j} E_{\bm\ell,r-2j}\rb{\frac Nn}^{\frac{r+3}{4}} + O\rb{n^{-\frac{R_0+3}{4}}}\right) + O\rb{n^{-\frac{R+3}4}e^{\pi\sqrt{\frac n3}}}.
	\end{align*}
	As the terms with $r\ge R$ also have size $O(n^{-\frac{R+3}{4}}e^{\pi\sqrt\frac n3})$, we may truncate the asymptotic expansion, and obtain
	\begin{equation*}
		\frac{1}{2\pi i}\int_{\mc C\rb{\frac{\pi N}{2\sqrt{3n}}}} g_{\a,\b;N,\bm\ell}^{[K]}(z)e^\frac{nz}{N} dz = \frac{e^{\pi \sqrt{\frac n3}}}{2^{K+\frac12}\pi^\frac N2}\left(\sum_{r=0}^{R-1} \sum_{j=0}^{\Flo{\frac r2}} c_{\frac r2-j,\frac\pi2\sqrt{\frac{N}{3}},j} E_{\bm\ell,r-2j}\rb{\frac Nn}^{\frac{r+3}{4}} + O\rb{n^{-\frac{R+3}{4}}}\right).
	\end{equation*}
	Plugging this back into \eqref{eq:dc} gives
	\[
		\scalebox{0.92}{$\displaystyle d_{\a,\b;N}^{[K]}(n) = \frac{e^{\pi\sqrt\frac n3}}{2^{K+\frac12}\pi^\frac N2N}\left(\sum_{r=0}^{R-1} \sum_{k=0}^{N-1} \sum_{\bm\ell\in(\Z/N\Z)^N} \z^{(n-NH(\bm\ell))k}_N\sum_{j=0}^{\Flo{\frac r2}} c_{\frac r2-j,\frac\pi2\sqrt\frac N3,j}E_{\bm\ell,r-2j}\left(\frac Nn\right)^\frac{r+3}{4} + O\left(n^{-\frac{R+3}{4}}\right)\right).$}
	\]
	The theorem follows, using orthogonality of roots of unity.
\end{proof}

\section{Proof of Theorems \ref{thm:main} and \ref{thm:main2}}\label{section:pot}

Now we are ready to prove \Cref{thm:main}.

\begin{proof}[Proof of \Cref{thm:main}]
	To prove \Cref{thm:main}, it suffices to determine the first two terms in \Cref{thm:da}. We compute
	\[
		c_{0,\frac\pi2\sqrt{\frac N3},0} =\frac{N^\frac14}{2\sqrt{2}\cdot3^\frac14},\qquad 
		c_{\frac12,\frac\pi2\sqrt{\frac N3},0} = \frac14\sqrt{\frac{\pi N}{3}}. 
	\]
	We follow the proof of \Cref{prp:gK1a} to determine $E_{\bm\ell,0}$ and $E_{\bm\ell,1}$. To begin with, we compute the coefficients $C_0(\bm u)$ and $C_1(\bm u)$ in the formal power series expansion \eqref{eq:Crud}. We expand
	\[
		\exp(\phi(\bm u,z)) = \exp\rb{- \bm b^T \bm u z^{\frac12}-\frac{Nz}{24} -\sum_{j=1}^N \sum_{r\geq2} \rb{B_r\rb{-\frac{u_j}{\sqrt{z}}} - \frac{\delta_{r,2}u_j^2}{z}} \Li_{2-r}\rb{\frac 12} \frac{z^{r-1}}{r!}}.
	\]
	Noting that we have
	\begin{align*}
		&\exp\left(-\sum_{j=1}^N \sum_{r\ge2} \left(B_r\left(-\frac{u_j}{\sqrt{z}}\right)-\frac{\d_{r,2}u_j^2}{z}\right) \Li_{2-r}\left(\frac12\right)\frac{z^{r-1}}{r!}\right) = 1+\sum_{j=1}^N \rb{-\frac{u_j}{2}+\frac{u_j^3}{3}}\sqrt{z} + O(z),\\
		&\exp\rb{- \bm b^T \bm u\sqrt{z}} = 1-\sum_{j=1}^N b_ju_j\sqrt{z}+O(z) = 1+\sum_{j=1}^N \rb{-\frac j N + \frac 1 2 - e_j}u_j\sqrt{z}+O(z),
	\end{align*}
	we get
	\[
		C_0(\bm u) = 1,\qquad C_1(\bm u) = \sum_{j=1}^N \left(-\left(\frac jN+e_j\right)u_j+\frac{u_j^3}{3}\right).
	\]
	
	Next we compute the constants $V_{0,-N}$, $V_{0,1-N}$, and $V_{1,-N}$. For this we use the evaluations
	\begin{align*}
		e^{-\frac{\pi^2N}{12z}}\int_{\Rc_{\a,\b;N}} G_0(\bm u) \bm{du} &= \int_{\Rc_{\a,\b;N}} e^{-\bm u^T\bm u} \bm{du} = \frac{\pi^\frac N2}{2},\\
		e^{-\frac{\pi^2N}{12z}}\int_{\R^{N-1}} G_0\left(\bm{u_{[1]}},u_\b\right) \bm{du_{[1]}} &= \int_{\R^{N-1}} e^{-u_\b^2-\bm{u_{[1]}}^T\bm{u_{[1]}}} \bm{du_{[1]}} = \frac{\pi^\frac{N-1}{2}}{\sqrt2},
	\end{align*}
	where $\Rc_{\a,\b;N}:=\{\bm u\in\R^N:u_\a\ge u_\b\}$. Meanwhile, we compute
	\begin{align}\nonumber
		e^{-\frac{\pi^2N}{12z}} \int_{\mc R_{\alpha, \beta;N}} G_1(\bm u) \bm{du} &= \int_{\mc R_{\alpha, \beta;N}} C_1(\bm u) e^{-\bm u^T\bm u} \bm{du}\\
		&= - \sum_{j=1}^N\rb{\frac jN + e_j} \int_{\mc R_{\alpha, \beta;N}} u_j e^{-\bm u^T\bm u} \bm{du} + \frac13\sum_{j=1}^N \int_{\mc R_{\alpha, \beta;N}} u_j^3 e^{-\bm u^T\bm u} \bm{du}. \label{eq:C1_split}
	\end{align}
	We consider the first integral on the right-hand side. If $j\not\in\{\a,\b\}$, then we have
	\ba
		\int_{\mc R_{\alpha, \beta;N}} u_j e^{-\bm u^T\bm u} \bm{du} = \int_{u_\alpha>u_\beta} e^{-u_\alpha^2-u_\beta^2} du_\alpha du_\beta \int_{\R^{N-2}} u_j e^{-\bm u^T\bm u} \bm{du} = 0
	\ea
	because the rightmost integral is anti-symmetric with respect to $u_j$. On the other hand, we have
	\begin{align*}
		\int_{\Rc_{\a,\beta;N}} u_\a e^{-\bm u^T\bm u} \bm{du} &= \int_{u_\a>u_\b} u_\a e^{-u_\a^2-u_\b^2} du_\a du_\b\int_{\R^{N-2}} e^{-\bm u^T\bm u} \bm{du} = \frac{\pi^\frac{N-1}{2}}{2\sqrt{2}},\\
		\int_{\Rc_{\a,\beta;N}} u_\b e^{-\bm u^T\bm u} \bm{du} &= \int_{u_\a>u_\b} u_\b e^{-u_\a^2-u_\b^2} du_\a du_\b\int_{\R^{N-2}} e^{-\bm u^T\bm u} \bm{du} = -\frac{\pi^\frac{N-1}{2}}{2\sqrt{2}}.
	\end{align*}
	The second integral on the right of \eqref{eq:C1_split} is invariant under interchanging $u_\a$ and $u_\b$. Hence
	\[
		\sum_{j=1}^N \int_{\Rc_{\a,\beta;N}} u_j^3e^{-\bm u^T\bm u} \bm{du} = \frac12\sum_{j=1}^N \int_{\R^N} u_j^3e^{-\bm u^T\bm u} \bm{du} = 0,
	\]
	since the integral is anti-symmetric. Plugging into \eqref{eq:C1_split}, it follows that
	\[
		e^{-\frac{\pi^2N}{12z}} \int_{\mc R_{\alpha, \beta;N}} G_1(\bm u) \bm{du} = \rb{\frac{\beta-\alpha}N + (e_\beta-e_\alpha)} \frac{\pi^{\frac{N-1}2}}{2\sqrt{2}}.
	\]
	It follows from \eqref{eq:VjN} and \eqref{eq:Vj1N} (and \eqref{eq:W0} for $W_0$) that
	\[
		V_{0,-N} = \frac{\pi^\frac N2}{2N^N},\quad V_{0,1-N} = \frac{\pi^\frac{N-1}{2}}{\sqrt2N^{N-1}}\left(\frac12-\frac{[\ell_\a-\ell_\b]_N}N\right),\quad V_{1,-N} = \frac{\pi^\frac{N-1}{2}(\b-\a+N(e_\b-e_\a))}{2\sqrt2N^{N+1}}.
	\]
	
	Finally, using \eqref{eq:Elrd}, we compute
	\ba
		E_{\bm\ell,0} &= V_{0,-N} = \frac{\pi^{\frac N2}}{2N^N},\\ 
		E_{\bm\ell,1} &= V_{0,1-N} + V_{1,-N} = \frac{\pi^\frac{N-1}{2}}{2\sqrt{2}N^{N-1}} \left(1-\frac{2[\ell_\a-\ell_\b]_N}{N}+\frac{\b-\a+N(e_\b-e_\a)}{N^2}\right).
	\ea
	\Cref{thm:main} then follows from \Cref{thm:da}.
\end{proof}

To prove \Cref{thm:main2}, we need the following lemma, which follows by a direct calculation.

\begin{lem}\label{lem:lc}
	Let $N\ge5$, $1\le\a,\b\le N$, and $r,\ell_\a,\ell_\b\in\Z/N\Z$. Then
	\ba
		\#\cb{\bm\ell_{\bm{[2]}}\in(\Z/N\Z)^{N-2} : NH\rb{\bm\ell_{\bm{[2]}},\ell_\alpha,\ell_\beta}\equiv r\pmod{N}} = N^{N-3},
	\ea
	where $\bm\ell_{\bm{[2]}}\in(\Z/N\Z)^{N-2}$ runs through the indices $j\not\in\{\alpha, \beta\}$. 
\end{lem}

Now we are ready to prove \Cref{thm:main2}.

\begin{proof}[Proof of \Cref{thm:main2}]
	The case $N=2$ can be verified directly from \Cref{thm:main}. Now suppose that $N\ge5$. By \Cref{lem:lc}, for every $r,\ell_\a,\ell_\b\in\Z/N\Z$, there exist $N^{N-3}$ tuples {$\bm\ell\in(\Z/N\Z)^N$} such that $NH(\bm\ell)\equiv r\Pmod N$. So we may evaluate the inner sum in \Cref{thm:main} as follows:
	\begin{align*}
		&\sum_{\substack{\bm\ell\in(\Z/N\Z)^N\\ NH(\bm\ell)\equiv r\pmod{N}}} \rb{1+\frac{N^2-2{ N}[\ell_\alpha-\ell_\beta]_N+ \beta-\alpha+ N\rb{e_\beta - e_\alpha}}{2\cdot 3^\frac14\sqrt{N}}n^{-\frac 14}}\\
		&\hspace{2cm}= N^{N-1}\rb{1+\frac{{-N}+\beta-\alpha+ N \rb{e_\beta - e_\alpha}}{2\cdot 3^\frac14\sqrt{N}}n^{-\frac 14}}.\qedhere
	\end{align*}
\end{proof}

\section{Numerical examples}\label{section:nx} 

We provide numerical data supporting our statements. All computations were done in PARI/GP and the plots were created in Sage \cite{PARI2,sage}.

\subsection{{Numerical data for $\bm{N=2}$}}

We provide some numerical data for the parity bias problem.

\begin{example}
	Let $N=2$, $\{\a,\b\}=\{1,2\}$, and $K\in\{0,1\}$, corresponding to the claims in \cite{BBDMS,KKL2020}.
	The numbers $d^{[K]}_{1,2;2}(n)$, $d^{[K]}_{2,1;2}(n)$ and their difference for $0\le n\le50$ are given in \Cref{table:k0} for $K=0$ and in \Cref{table:k1} for $K=1$. The differences $d^{[K]}_{1,2;2}(n)-d^{[K]}_{2,1;2}(n)$ are plotted for $0\le n\le100$ in \Cref{figureN2}. For $K=0$ we observe, in accordance with \cite{BBDMS,KKL2020}, that $d^{[0]}_{1,2;2}(n)>d^{[0]}_{2,1;2}(n)$ for $n\ge20$. For $K=1$, we note that the numbers for $13\le n\le29$ agree with the claim in \cite{BBDMS}, Problem 6.1, i.e.,
	\begin{equation*}
		d^{{[1]}}_{1,2;2}(n)-d^{{[1]}}_{2,1;2}(n)
		\begin{cases}
			>0 &\text{if $n$ is even,}\\
			<0 &\text{if $n$ is odd},
		\end{cases}
	\end{equation*}
	whereas numerics suggests that for all $n\ge29$ we have $d_{1,2;2}^{[1]}(n)-d_{2,1;2}^{[1]}(n)<0$. This was checked numerically up to $n\le10000$ which took $5$ minutes using PARI/GP on an Apple M1 Pro chip.
	
	\begin{figure}[H]
		\includegraphics[width=0.68\textwidth]{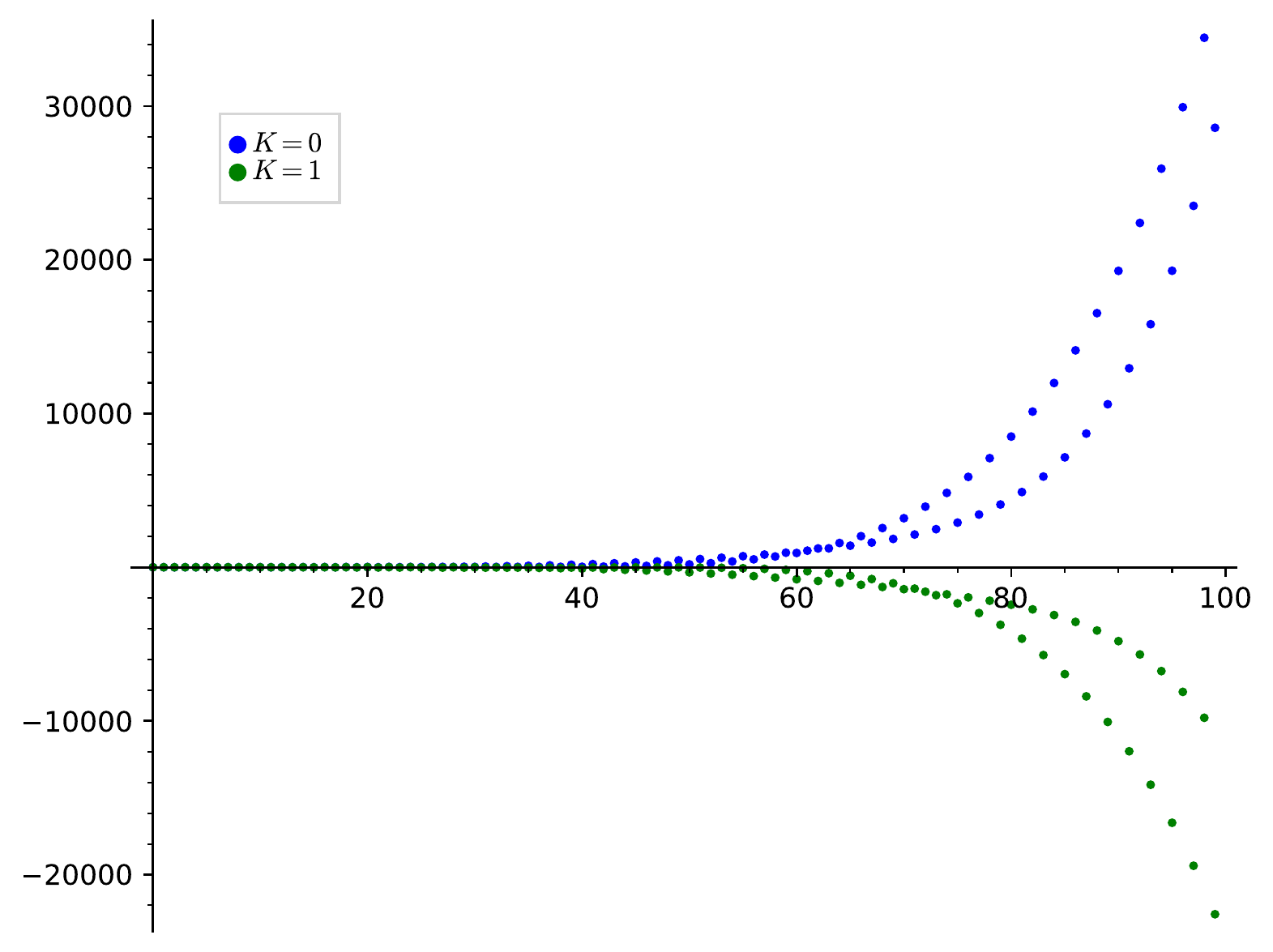}
		\caption{$d^{[K]}_{1,2;2}(n)-d^{[K]}_{2,1;2}(n)$ for $0\leq n\leq 100$, $K\in\{0,1\}$}
		\centering
		\label{figureN2}
	\end{figure}

	Furthermore, the numbers $(d_{1,2;2}^{[K]}(n)-d_{2,1;2}^{[K]}(n))ne^{-\pi\sqrt\frac n3}$ are plotted in \Cref{figureN2_ratio} for $10\le n\le5000$. Supporting \Cref{cor:pb}, the figure suggests that they converge to $(-1)^K2^{\frac72-K}3^{-\frac12}$, independently of $n\Pmod2$.
	
	\begin{figure}[H]
		\includegraphics[width=0.68\textwidth]{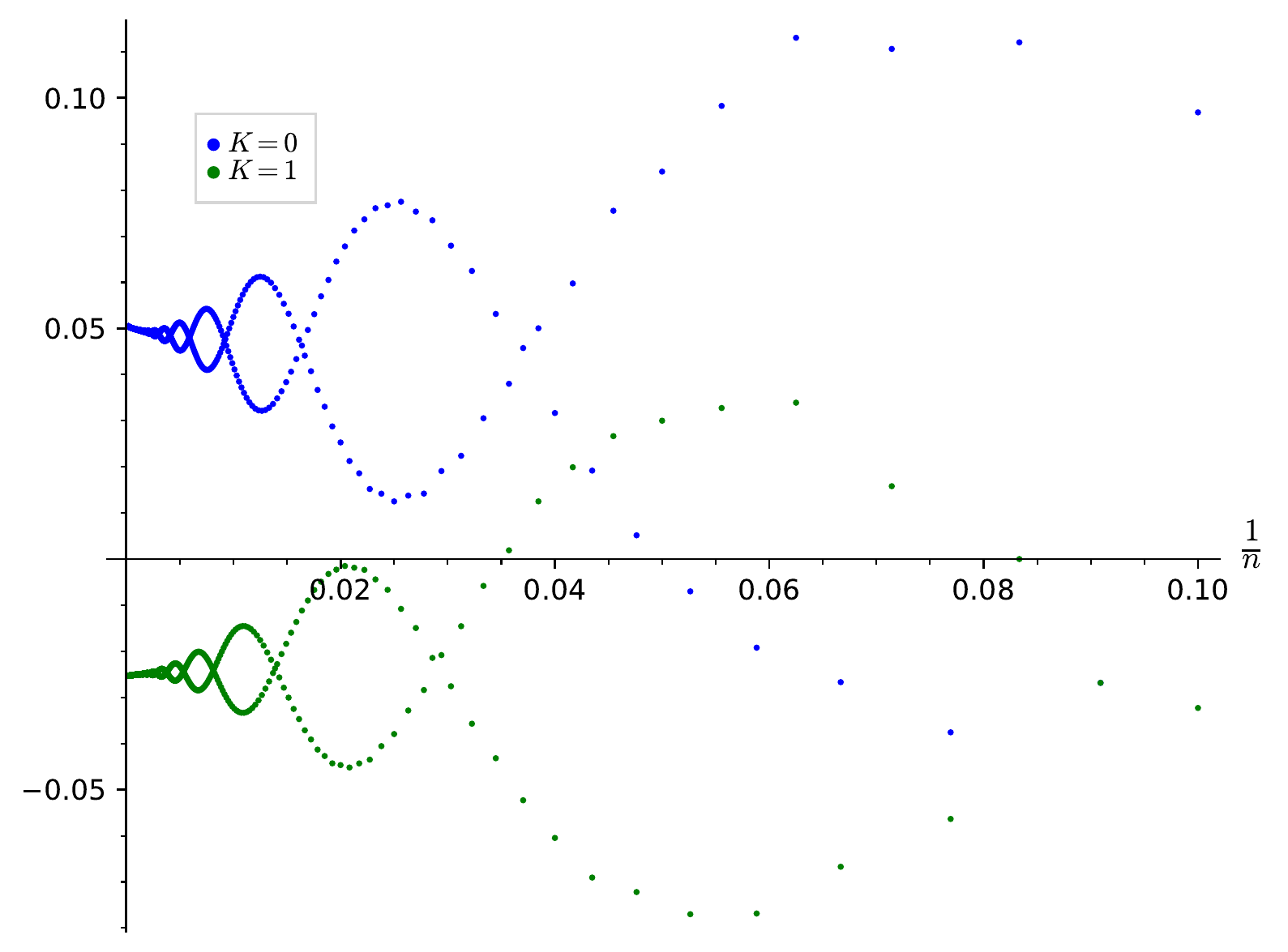}
		\caption{$(\frac1n,\; (d_{1,2;2}^{[K]}(n) - d_{2,1;2}^{[K]}(n) ) n e^{-\pi\sqrt{\frac n3}})$ for $10\leq n\leq 5000$, $K\in\{0,1\}$}
		\centering
		\label{figureN2_ratio}
	\end{figure}
	
	\centering
	\begin{table}[H]
		\begin{tabular}{c | cccccccccccccccccccc}
			$n$ & $1$ & $2$ & $3$ & $4$ & $5$ & $6$ & $7$ & $8$ & $9$ & $10$ & $11$ & $12$ & $13$ & $14$ & $15$ & $16$ & $17$ & $18$ \\\hline
			\rule{0pt}{5mm}
			$d^{[0]}_{1,2;2}(n)$ & $1$ & $0$ & $1$ & $1$ & $1$ & $2$ & $1$ & $4$ & $2$ & $6$ & $3$ & $9$ & $5$ & $12$ & $9$ & $17$ & $14$ & $22$ \\
			$d^{[0]}_{2,1;2}(n)$ & $0$ & $1$ & $0$ & $1$ & $0$ & $2$ & $1$ & $2$ & $2$ & $3$ & $4$ & $4$ & $7$ & $5$ & $11$ & $7$ & $16$ & $10$ \\
			$d^{[0]}_{1,2;2}(n)-d^{[0]}_{2,1;2}(n)$& $1$ & $-1$ & $1$ & $0$ & $1$ & $0$ & $0$ & $2$ & $0$ & $3$ & $-1$ & $5$ & $-2$ & $7$ & $-2$ & $10$ & $-2$ & $12$ \\
		\end{tabular}\vspace{20pt}
		\begin{tabular}{ccccccccccccccccc}
			$19$ & $20$ & $21$ & $22$ & $23$ & $24$ & $25$ & $26$ & $27$ & $28$ & $29$ & $30$ & $31$ & $32$ & $33$ & $34$ & $35$ \\\hline
			$22$ & $29$ & $33$ & $38$ & $48$ & $50$ & $68$ & $65$ & $95$ & $86$ & $128$ & $113$ & $172$ & $149$ & $226$ & $197$ & $295$ \\
			$23$ & $15$ & $32$ & $21$ & $43$ & $32$ & $57$ & $45$ & $74$ & $66$ & $96$ & $92$ & $123$ & $129$ & $157$ & $175$ & $199$ \\
			$-1$ & $14$ & $1$ & $17$ & $5$ & $18$ & $11$ & $20$ & $21$ & $20$ & $32$ & $21$ & $49$ & $20$ & $69$ & $22$ & $96$ \\
		\end{tabular}\vspace{20pt}
		\begin{tabular}{ccccccccccccccc}
			$36$ & $37$ & $38$ & $39$ & $40$ & $41$ & $42$ & $43$ & $44$ & $45$ & $46$ & $47$ & $48$ & $49$ & $50$ \\\hline
			$260$ & $379$ & $342$ & $485$ & $449$ & $613$ & $587$ & $773$ & $762$ & $967$ & $987$ & $1206$ & $1269$ & $1497$ & $1623$ \\
			$239$ & $253$ & $316$ & $320$ & $419$ & $406$ & $544$ & $514$ & $704$ & $652$ & $898$ & $825$ & $1142$ & $1045$ & $1435$ \\
			$21$ & $126$ & $26$ & $165$ & $30$ & $207$ & $43$ & $259$ & $58$ & $315$ & $89$ & $381$ & $127$ & $452$ & $188$ \\
		\end{tabular}
		\caption{Numerics for $K=0$.}
		\label{table:k0}
	\end{table}

	\newpage

	\begin{table}[H]
		\begin{tabular}{c | ccccccccccccccccccc}
			$n$ & $1$ & $2$ & $3$ & $4$ & $5$ & $6$ & $7$ & $8$ & $9$ & $10$ & $11$ & $12$ & $13$ & $14$ & $15$ & $16$ & $17$ \\\hline
			\rule{0pt}{5mm}
			$d^{{[1]}}_{1,2;2}(n)$ & $0$ & $0$ & $1$ & $0$ & $1$ & $0$ & $1$ & $1$ & $1$ & $2$ & $1$ & $4$ & $1$ & $6$ & $2$ & $9$ & $3$ \\
			$d^{{[1]}}_{2,1;2}(n)$ & $0$ & $1$ & $0$ & $1$ & $0$ & $2$ & $0$ & $2$ & $1$ & $3$ & $2$ & $4$ & $4$ & $5$ & $7$ & $6$ & $11$\\
			$d^{[1]}_{1,2;2}(n)-d^{[1]}_{2,1;2}(n)$ & $0$ & $-1$ & $1$ & $-1$ & $1$ & $-2$ & $1$ & $-1$ & $0$ & $-1$ & $-1$ & $0$ & $-3$ & $1$ & $-5$ & $3$ & $-8$ \\
		\end{tabular}\vspace{20pt}
		\begin{tabular}{cccccccccccccccccc}
			$18$ & $19$ & $20$ & $21$ & $22$ & $23$ & $24$ & $25$ & $26$ & $27$ & $28$ & $29$ & $30$ & $31$ & $32$ & $33$ & $34$ & $35$ \\\hline
			$12$ & $5$ & $16$ & $9$ & $20$ & $14$ & $26$ & $22$ & $32$ & $33$ & $40$ & $48$ & $50$ & $67$ & $63$ & $93$ & $79$ & $125$ \\
			$8$ & $16$ & $11$ & $23$ & $14$ & $32$ & $20$ & $43$ & $27$ & $57$ & $39$ & $74$ & $54$ & $95$ & $76$ & $121$ & $103$ & $153$ \\
			$4$ & $-11$ & $5$ & $-14$ & $6$ & $-18$ & $6$ & $-21$ & $5$ & $-24$ & $1$ & $-26$ & $-4$ & $-28$ & $-13$ & $-28$ & $-24$ & $-28$ \\
		\end{tabular}\vspace{20pt}
		\begin{tabular}{ccccccccccccccc}
			$36$ & $37$ & $38$ & $39$ & $40$ & $41$ & $42$ & $43$ & $44$ & $45$ & $46$ & $47$ & $48$ & $49$ & $50$ \\\hline
			$101$ & $166$ & $129$ & $216$ & $166$ & $279$ & $215$ & $354$ & $278$ & $448$ & $360$ & $559$ & $467$ & $695$ & $603$ \\
			$143$ & $191$ & $191$ & $239$ & $257$ & $297$ & $338$ & $369$ & $444$ & $458$ & $572$ & $569$ & $737$ & $705$ & $935$ \\
			$-42$ & $-25$ & $-62$ & $-23$ & $-91$ & $-18$ & $-123$ & $-15$ & $-166$ & $-10$ & $-212$ & $-10$ & $-270$ & $-10$ & $-332$ \\
		\end{tabular}
		\caption{Numerics for $K=1$.}
		\label{table:k1}
	\end{table}
\end{example}

\subsection{{Numerical data for $\bm{N=3}$}}

We give numerical data to illustrate Corollary \ref{cor:N3}.

\begin{example}
	We consider $N=3$ and $K=0$. The first $17$ values for $d_{1,2;3}^{[0]}(n)$ and $d_{1,2;3}^{[0]}(n)$, and their difference are listed in \Cref{table:N3k0}. \Cref{figureN3} depicts the difference $d_{1,2;3}^{[0]}(n)-d_{2,1;3}^{[0]}(n)$ for $0\le n\le100$. Moreover, the numbers $(d_{1,2;3}^{[0]}(n)-d_{2,1;3}^{[0]}(n))ne^{-\pi\sqrt\frac n3}$ are plotted for $10\le n\le1000$ in \Cref{figureratN3}. As pointed out above, we observe that the asymptotics of the difference indeed depend on $n\Pmod3$.

	\begin{table}[H]
		\begin{tabular}{c | cccccccccccccccccccc}
			$n$ & $1$ & $2$ & $3$ & $4$ & $5$ & $6$ & $7$ & $8$ & $9$ & $10$ & $11$ & $12$ & $13$ & $14$ & $15$ & $16$ & $17$ \\\hline
			\rule{0pt}{5mm}
			$d^{{[0]}}_{1,2;3}(n)$ & $1$ & $0$ & $0$ & $2$ & $1$ & $0$ & $4$ & $2$ & $0$ & $8$ & $4$ & $1$ & $14$ & $8$ & $2$ & $24$ & $14$ \\
			$d^{{[0]}}_{2,1;3}(n)$ & $0$ & $1$ & $0$ & $0$ & $2$ & $0$ & $1$ & $4$ & $0$ & $2$ & $8$ & $0$ & $4$ & $14$ & $1$ & $8$ & $24$ \\
			$d^{{[0]}}_{1,2;3}(n)-d^{{[0]}}_{2,1;3}(n)$ & $1$ & $-1$ & $0$ & $2$ & $-1$ & $0$ & $3$ & $-2$ & $0$ & $6$ & $-4$ & $1$ & $10$ & $-6$ & $1$ & $16$ & $-10$ \\
		\end{tabular}
		\caption{Numerics for $N=3$, $K=0$.}
		\label{table:N3k0}
	\end{table}
\end{example}

\newpage

\begin{figure}[H]
	\includegraphics[width=0.7\textwidth]{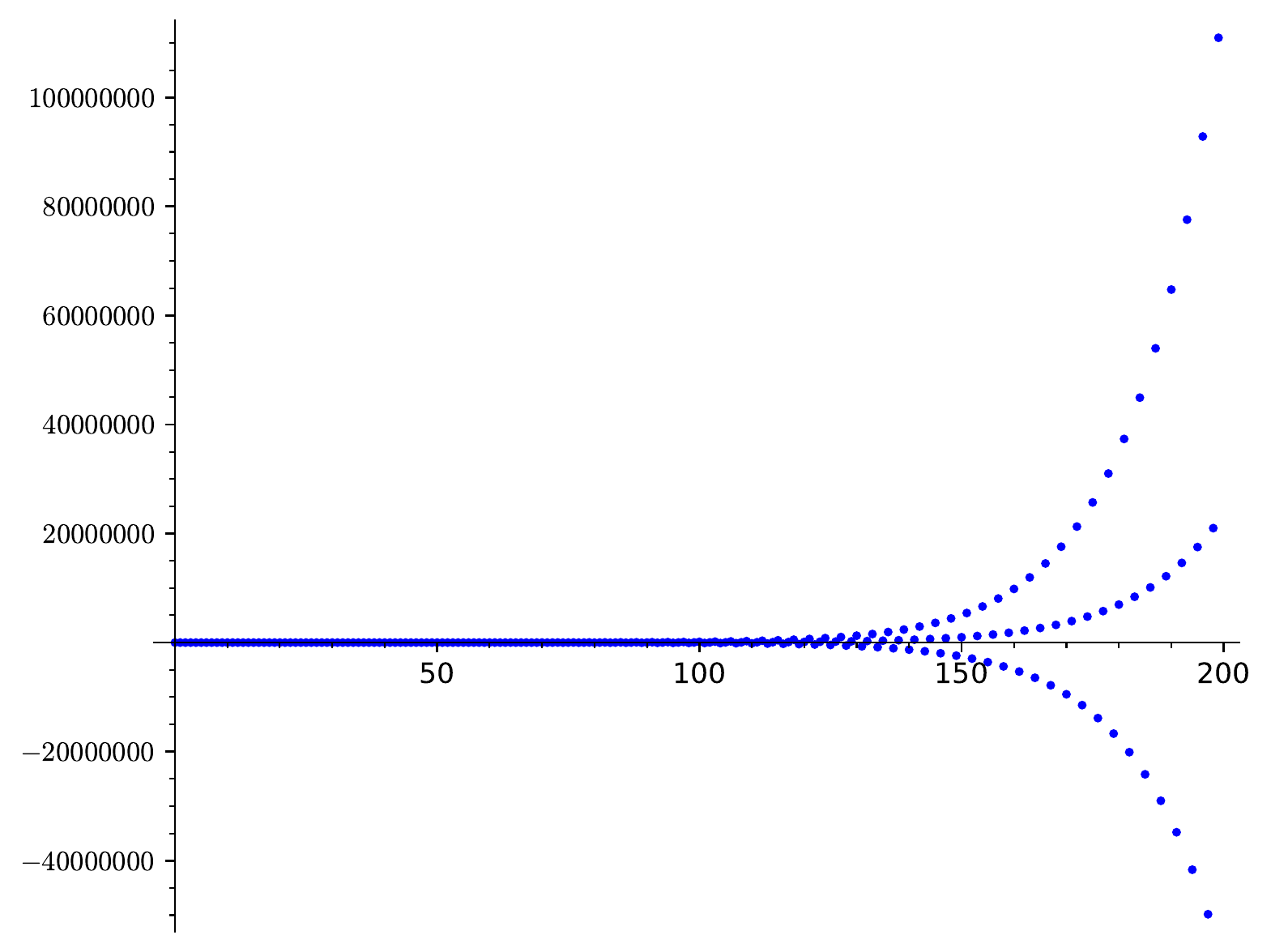}
	\caption{$d^{{[0]}}_{1,2;3}(n)-d^{{[0]}}_{2,1;3}(n)$ for $0\leq n\leq 100$}
	\centering
	\label{figureN3}
\end{figure}

\begin{figure}[H]
	\includegraphics[width=0.7\textwidth]{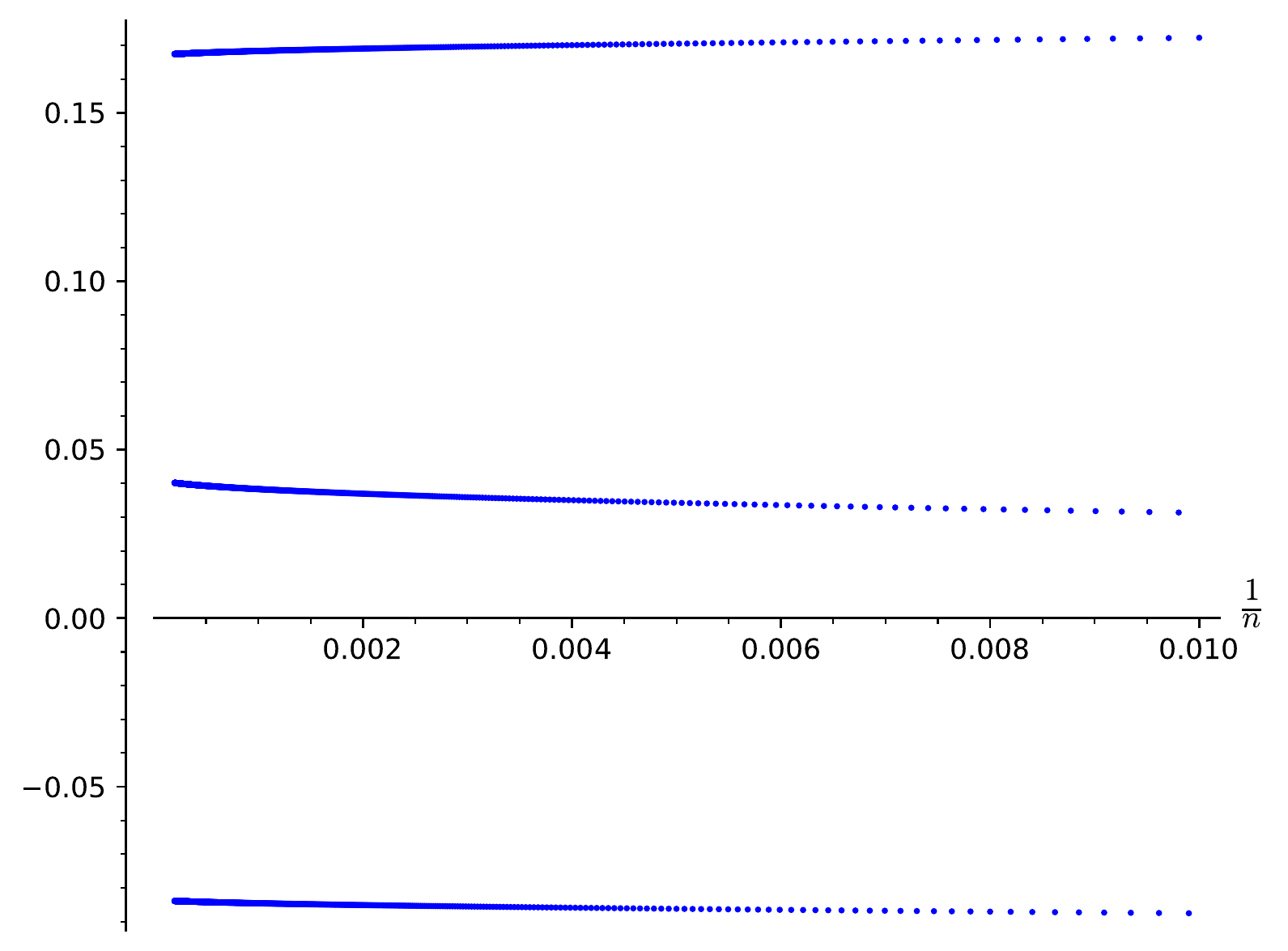}
	\caption{$(\frac1n,\; (d_{1,2;3}^{[0]}(n) - d_{2,1;3}^{[0]}(n) ) n e^{-\pi\sqrt{\frac n3}})$ for $10\leq n\leq 1000$}
	\centering
	\label{figureratN3}
\end{figure}

\end{document}